\RequirePackage{amsthm}%

\documentclass[sn-mathphys-num]{sn-jnl}

\usepackage{graphicx}%
\usepackage{multirow}%
\usepackage{amsmath,amssymb,amsfonts}%
\usepackage{mathrsfs}%
\usepackage[title]{appendix}%
\usepackage{xcolor}%
\usepackage{textcomp}%
\usepackage{manyfoot}%
\usepackage{booktabs}%
\usepackage{algorithm}%
\usepackage{algorithmicx}%
\usepackage{algpseudocode}%
\usepackage{listings}%
\usepackage{tikz,tikz-cd}%

\theoremstyle{thmstyleone}%
\newtheorem{theorem}{Theorem}
\newtheorem{proposition}[theorem]{Proposition}%
\newtheorem{lemma}[theorem]{Lemma}%
\newtheorem{corollary}[theorem]{Corollary}%

\theoremstyle{thmstyletwo}%
\newtheorem{remark}{Remark}%

\theoremstyle{thmstylethree}%
\newtheorem{definition}{Definition}%

\usetikzlibrary{decorations.pathmorphing}

\newcommand{\ie}{\textit{i.e.}}
\newcommand{\Grobner}{Gr\"obner}
\newcommand{\RnonNeg}{\ensuremath{\mathbb{R}_{\geq 0}}}
\newcommand{\Rpos}{\ensuremath{\mathbb{R}_{> 0}}}
\newcommand{\R}{\ensuremath{\mathbb{R}}}
\newcommand{\abs}[1]{\ensuremath{\left\vert#1\right\vert}}
\newcommand{\val}[1]{\ensuremath{\mathrm{val}\left(#1\right)}}
\newcommand{\dist}{\ensuremath{\mathrm{d}}}
\newcommand{\B}{\ensuremath{\mathrm{B}}}
\newcommand{\N}{\ensuremath{\mathbb{N}}}
\newcommand{\K}{\ensuremath{\mathbb{K}}}
\newcommand{\x}{\ensuremath{\mathbf{x}}}
\newcommand{\X}{\ensuremath{\DOTS{x_1}{x_n}}}
\newcommand{\KX}{\ensuremath{\K[\X]}}
\newcommand{\Kx}{\ensuremath{\K[\x]}}
\newcommand{\KXX}{\ensuremath{\K[[\X]]}}
\newcommand{\Kxx}{\ensuremath{\K[[\x]]}}
\newcommand{\sys}{\ensuremath{\mathfrak{X}^{<}_{R}}}
\newcommand{\NF}[1]{\ensuremath{\mathrm{NF}\left(#1\right)}}
\newcommand{\supp}[1]{\ensuremath{\mathrm{supp}\left(#1\right)}}
\newcommand{\coeff}[2]{\ensuremath{\left\langle#1\,\middle\vert\,#2\right\rangle}}
\newcommand{\op}[1]{\ensuremath{{#1}^{\mathrm{op}}}}
\newcommand{\LM}[2]{\ensuremath{\mathrm{lm}\left(#2\right)}}
\newcommand{\LC}[2]{\ensuremath{\mathrm{lc}\left(#2\right)}}

\newcommand{\ens}[2]{\ensuremath{\left\lbrace#1\,\middle\vert\,#2\right\rbrace}}
\newcommand{\DOTS}[3][,]{\ensuremath{{#2}#1{\cdots}#1{#3}}}
\newcommand{\functionDefinition}[5]{\ensuremath{
    \begin{tabular}{rccl}
        \ensuremath{#1\colon} & \ensuremath{#2} & $\longrightarrow$ &
            \ensuremath{#3} \\
                              & \ensuremath{#4} & $\longmapsto$     &
            \ensuremath{#5}
    \end{tabular} 
}}

\newcommand{\dTop}[1]{\ensuremath{\tau_{#1}^{\mathrm{dis}}}}

\newcommand\tikzcdcirclelength{6.75pt}
\newcommand\tikzcdcirclesep{-3.75pt}
\newcommand\tikzcddefaultlinewidth{0.4pt}
\tikzcdset{
  tikzcd to semithick/.tip={cm bold to},
  tikzcd to thick/.tip={cm to[scale=0.5]},
  tikzcd to
  circle/.tip={Circle[open,length={\tikzcdcirclelength},sep={\tikzcdcirclesep}]
  cm to},
  tikzcd to circle semithick/.tip={
    Circle[
      open,
      length={\tikzcdcirclelength+2*(0.6pt-\tikzcddefaultlinewidth)},
      sep={\tikzcdcirclesep-0.6pt+\tikzcddefaultlinewidth},
    ]
    cm to[scale=0.667]},
  tikzcd to circle thick/.tip={%
    Circle[
      open,
      length={\tikzcdcirclelength+2*(0.8pt-\tikzcddefaultlinewidth)},
      sep={\tikzcdcirclesep-0.8pt+\tikzcddefaultlinewidth},
    ]
    cm to[scale=0.5]},
  tikzcd left to
  circle/.tip={Circle[open,length={\tikzcdcirclelength},sep={\tikzcdcirclesep}]
  cm to[harpoon]},
  tikzcd right to
  circle/.tip={Circle[open,length={\tikzcdcirclelength},sep={\tikzcdcirclesep}]
  cm to[harpoon,swap]},
  rightarrow semithick/.code={\pgfsetarrows{tikzcd cap-tikzcd to
  semithick}\tikzset{semithick}},
  rightarrow thick/.code={\pgfsetarrows{tikzcd cap-tikzcd to
  thick}\tikzset{thick}},
  circle/.code={\pgfsetarrowsend{tikzcd to circle}},
  circle semithick/.code={\pgfsetarrowsend{tikzcd to circle
  semithick}\tikzset{semithick}},
  circle thick/.code={\pgfsetarrowsend{tikzcd to circle thick}\tikzset{thick}},
  circle harpoon/.code={\pgfsetarrowsend{tikzcd left to circle}},
  circle harpoon'/.code={\pgfsetarrowsend{tikzcd right to circle}},
  leftrightcircle/.code={\pgfsetarrows{tikzcd to circle-tikzcd to circle}},
  leftrightcircle harpoon/.code={\pgfsetarrows{tikzcd right to circle-tikzcd
    left to circle}},
  weird arrow/.code={\pgfsetarrows{tikzcd to-tikzcd to circle}},
}

\newcommand{\rew}[1][]{%
    \begin{tikzcd}[ampersand replacement=\&,cramped,sep=small]
        \ar{r}[']{#1}\&{}
    \end{tikzcd}}

\newcommand{\reflRew}[1][]{%
    \begin{tikzcd}[ampersand replacement=\&,cramped,sep=small]
        \ar{r}{=}[']{#1}\&{}
    \end{tikzcd}}

\newcommand{\transReflRew}[1][]{%
    \begin{tikzcd}[ampersand replacement=\&,cramped,sep=small]
        \ar{r}{\star}[']{#1}\&{}
    \end{tikzcd}}

\newcommand{\topRew}[1][]{%
    \begin{tikzcd}[ampersand replacement=\&,cramped,sep=small]
        \ar[circle]{r}[']{#1}\&{}
    \end{tikzcd}}
\newcommand{\topInvRew}[1][]{%
    \begin{tikzcd}[ampersand replacement=\&,cramped,sep=small]
        {}\&\ar[circle]{l}{#1}
    \end{tikzcd}}

\newcommand{\topSymRew}[1][]{%
    \begin{tikzcd}[ampersand replacement=\&,cramped,sep=normal]
        \ar[leftrightcircle]{r}[']{#1}\&{}
    \end{tikzcd}}
\newcommand{\topEquivRew}[1][]{%
    \begin{tikzcd}[ampersand replacement=\&,cramped,sep=normal]
        \ar[leftrightcircle]{r}{\star}[']{#1}\&{}
    \end{tikzcd}}

\newcommand{\topChainsRew}[1][]{%
    \begin{tikzcd}[ampersand replacement=\&,cramped,sep=small]
        \ar[circle, tail]{r}[']{#1}\&{}
    \end{tikzcd}}
\newcommand{\topChainsInvRew}[1][]{%
    \begin{tikzcd}[ampersand replacement=\&,cramped,sep=small]
        {}\&\ar[circle, tail]{l}{#1}
    \end{tikzcd}}

\newcommand{\topChainsSymRew}[1][]{%
    \begin{tikzcd}[ampersand replacement=\&,cramped,sep=normal]
        \ar[leftrightcircle]{r}[marking]{\diamond}[']{#1}\&{}
    \end{tikzcd}}
\newcommand{\topChainsEquivRew}[1][]{%
    \begin{tikzcd}[ampersand replacement=\&,cramped,sep=normal]
        \ar[leftrightcircle]{r}{\star}[marking]{\diamond}[']{#1}\&{}
    \end{tikzcd}}

\newcommand{\squigRew}[1][]{%
    \begin{tikzcd}[ampersand replacement=\&,cramped,sep=small]
        \ar[squiggly]{r}[']{#1}\&{}
    \end{tikzcd}}
\newcommand{\squigInvRew}[1][]{%
    \begin{tikzcd}[ampersand replacement=\&,cramped,sep=small]
        {}\&\ar[squiggly]{l}{#1}
    \end{tikzcd}}
\newcommand{\squigReflRew}[1][]{%
    \begin{tikzcd}[ampersand replacement=\&,cramped,sep=small]
        \ar[squiggly]{r}{=}\&{}
    \end{tikzcd}}
\newcommand{\squigReflInvRew}[1][]{%
    \begin{tikzcd}[ampersand replacement=\&,cramped,sep=small]
        {}\&\ar[squiggly]{l}[']{=}
    \end{tikzcd}}

\newcommand{\squigTransReflRew}[1][]{%
    \begin{tikzcd}[ampersand replacement=\&,cramped,sep=small]
        \ar[squiggly]{r}{\star}[']{#1}\&{}
    \end{tikzcd}}
\newcommand{\squigTransReflInvRew}[1][]{%
    \begin{tikzcd}[ampersand replacement=\&,cramped,sep=small]
        {}\&\ar[squiggly]{l}[']{\star}{#1}
    \end{tikzcd}}
\newcommand{\squigSymRew}[1][]{%
    \begin{tikzcd}[ampersand replacement=\&,cramped,sep=small]
        \ar[squiggly,leftrightarrow]{r}[']{#1}\&{}
    \end{tikzcd}}
\newcommand{\squigEquivRew}[1][]{%
    \begin{tikzcd}[ampersand replacement=\&,cramped,sep=small]
        \ar[squiggly,leftrightarrow]{r}{\star}[']{#1}\&{}
    \end{tikzcd}}

\begin{document}

\title[Computing in complete local equicharacteristic Noetherian rings via
topological rewriting on commutative formal power series]{Computing in complete
local equicharacteristic Noetherian rings via topological rewriting on
commutative formal power series}


\author*{\fnm{Adya} \sur{Musson-Leymarie}}
\email{adya.musson-leymarie@unilim.fr}

\affil{Univ. Limoges, CNRS, XLIM, UMR 7252, F-87000 Limoges}


\abstract{In commutative algebra, the theory of \Grobner{} bases enables one to
    compute in any finitely generated algebra over a given computable field. For
    non-finitely generated algebras however, other methods have to be pursued.
    For instance, it follows from the Cohen structure theorem that standard
    bases of formal power series ideals offer a similar prospect but for
    complete local equicharacteristic rings whose residue field is computable.
    Using the language of rewriting theory, one can characterise \Grobner{}
    bases in terms of confluence of the induced rewriting system. It has been
    shown, so far via purely algebraic tools, that an analogous characterisation
    holds for standard bases with a generalised notion of confluence.
    Subsequently, that result is utilised to prove that two generalised
    confluence properties, where one is actually in general strictly stronger
    than the other, are actually equivalent in the context of formal power
    series. In the present paper, we propose alternative proofs making use of
    tools purely from the new theory of topological rewriting to recover both
    the characterisation of standard bases and the equivalence between
    generalised confluence properties. The objective is to extend the analogy
    between \Grobner{} basis theory together with classical algebraic rewriting
    theory and standard basis theory with topological rewriting theory.}



\maketitle

\section*{Introduction}\label{sct:intro}
\addcontentsline{toc}{section}{Introduction}

In abstract algebra, one defines algebraic structures in terms of axiomatic
properties of operations. No reference is made to the nature of the elements of
the underlying set(s). Hence, among examples of associative commutative algebras
over a field, say, for instance, the field $\R$ of real numbers, one might
consider simple examples, such as the algebra of real numbers itself or the
$\R$-algebra of complex numbers, but one could also look at more abstract and
intricate sets of elements that can be equipped with a real algebra structure,
like, for example, the $\R$-algebra of smooth functions on an open subset of a
differentiable manifold. As mathematicians, the ability to compute, in the sense
of making calculations based on the operations in an algebraic structure, is
often central in obtaining new results. Even more, physicists, engineers and
other applied scientists require to make effective computations to solve their
problems. One big question then arises: is it enough to make approximations of
the desired result, assuming it is possible to estimate how far off the
approximation actually is from the answer, or do we absolutely require an exact
solution of the problem? The two approaches have their own respective merit and
which one is to be preferred strongly depends on the application. Approximations
are the objects of study in \emph{numerical analysis}. In this paper, we are
concerned with computing exact solutions, even if they might be just symbolic
expressions; this is the domain of \emph{computer algebra} in which one
implements algorithms on computers to perform effective symbolic computations so
it yields an exact answer.

Just as \emph{representation theory} offers a way of computing in even the more
abstract groups by effectively looking at its elements as invertible matrices,
that is, as a finite two-dimensional array of elements of a given field,
associative algebras over a field $\K$ are always isomorphic to the quotient
algebra of a \emph{free $\K$-algebra} modulo one of its ideals. The advantage is
that free algebras are combinatorial and syntactic objects that we know how to
effectively represent on paper or on a computer. Here, we are concerned only with
commutative algebras and, therefore, free $\K$-algebras take the form of the
well-known polynomial algebras with coefficients in $\K$. We can thus take for
granted: we know how to compute in polynomial algebras if we know how to
calculate in $\K$. Now, our problem of computing in the starting associative
algebra is not completely solved, as we want to work in a quotient algebra of a
polynomial algebra. Thankfully, some powerful machinery to help us with that
goal has been developed during the second half of the last century, namely
\emph{commutative \Grobner{} bases}. Their first concrete formulation as we know
them today is attributed to \textsc{Buchberger} in his PhD
thesis~\cite{Buchberger_2006}. \Grobner{} bases and their generalisations
actually have a multitude of very useful applications, especially in commutative
algebra and algebraic geometry (see for instance~\cite{Buchberger_1998,
Greuel_2002}).

The reason why polynomial algebras are of such importance is because of two
things: their syntactic nature and their \emph{universal property} among other
associative commutative algebras, \ie{} that property we mentioned about any
algebra being isomorphic to a quotient of a polynomial algebra. However,
consider a $\K$-algebra $A$ that is not finitely generated, that is to say,
there is no finite subset $\{\DOTS{x_1}{x_n}\} \subseteq A$ that generates $A$
as a $\K$-algebra, \ie{} such that the canonical morphism $\KX \rightarrow A$ is
surjective. This means that for such an algebra $A$, we have to consider
infinitely many indeterminates in our polynomial algebra and, it turns out, many
important results from \Grobner{} basis theory do not apply anymore. However,
alternative tools are available for certain classes of non-finitely generated
$\K$-algebras. Indeed, even before \textsc{Buchberger}, \textsc{Hironaka} had
developed in~\cite{Hironaka_1964, Hironaka_1964a} the notion of \emph{standard
basis} of an ideal of commutative formal power series when he was interested in
the problem of \emph{resolution of singularities} in algebraic geometry (for
more details, see~\cite{Hauser_2003, Spivakovsky_2020}).  As stated in
Proposition~\ref{prop:Kxx-not-finitely-generated}, commutative formal power
series with coefficients in $\K$ are not finitely generated as $\K$-algebras and
therefore, standard bases are useful in that case, where \Grobner{} bases fail
to help.

Now, the \emph{Cohen structure theorem} originally proven in~\cite{Cohen_1946}
gives us insight into a certain class of rings that can be equipped with
the structure of an algebra over an appropriate field. We recall that result in
Theorem~\ref{thm:cohen-structure}. In that theorem,
Statement~\ref{item:Kxx-free-algebra} expresses the kind of universal property
for commutative formal power series that polynomial algebras enjoy but for the
specific class of \emph{complete local equicharacteristic Noetherian rings}.
Those rings actually contain a subfield $\K$, which turns out to be isomorphic
to their so-called \emph{residue field}, and, when seen as $\K$-algebras, they
are isomorphic to a quotient of a commutative formal power series algebra in
finitely many indeterminates. Moreover, Statement~\ref{item:sub-algebra}
enables us to prove
Proposition~\ref{prop:comp-loc-equi-noeth-ring-not-finitely-generated} which
states that even those complete local equicharacteristic Noetherian rings are
not finitely generated when considered as algebras over their residue field and,
therefore, \Grobner{} bases are of no use. Among examples of complete local
equicharacteristic Noetherian rings, one can find commutative formal power
series rings, and complete local Cohen-Macaulay rings to name a
few.

So far, we have given motivations for the notion of standard basis by comparing
it with \Grobner{} bases. However, unlike for \Grobner{} bases and so-called
algebraic or linear rewriting theory, there is a lack of literature on
the subject of standard bases viewed from \emph{rewriting theory}. But first,
what is rewriting theory? It is a language to formalise the concept of sequences
of computations. It originally comes from group theory in order to solve
decision problems and it then got adapted in computer science and, more
precisely, to fulfill the needs to formalise and unify certain models of
computations in computability theory as well as give explicit formal systems in
logic. Rewriting theory suffers from a lack of easily accessible literature on
the topic if it were not for the two standard references~\cite{Baader_1999,
Terese_2003}. Now, despite rewriting theory being originally motivated by its
use in group theory, logic, term rewriting and lambda-calculus, it has proven to
be exceptionally powerful at providing a language for the theory of \Grobner{}
bases (see~\cite{Bergman_1977, Winkler_1989}). In that context, a \Grobner{}
basis is characterised by the so-called \emph{confluence} property of the
rewriting system it induces, which exactly means that the system enjoys an
underlying determinism. More precisely, this confluence property states that, no
matter what path of sequences of computation one chooses to reduce the
polynomial, one always reaches the same result, called the \emph{normal form of
the polynomial}.  Such normal forms can be used as canonical representatives of
the equivalence classes in the quotient algebra.  This allows to keep our
computations in the polynomial algebra, where we know how to calculate, and, at
each step, reducing to normal form to get the corresponding equivalence class.

As we mentioned however, standard bases had so far not been investigated from a
rewriting-theoretic perspective, strictly speaking. Only recently
in~\cite{Chenavier_2020}, \textsc{Chenavier} used algebraic tools to
characterise standard bases of commutative formal power series with the notion
of $\delta$-confluence, a generalisation of the classical confluence property.
The author also introduced a generalisation of the transitive
reflexive closure of the one-step reduction relation which is studied in
classical rewriting theory. That generalised relation, called the
\emph{topological rewriting relation}, depends on a topology on the underlying
set of the rewriting system and, therefore, that new theory is called
\emph{topological rewriting theory}. It provides a framework in which one
can investigate more efficiently rewriting systems which fail to be confluent or
\emph{terminating} in the classical sense. This includes formal power series for
the algebraists but also, infinitary term-rewriting and infinitary
lambda-calculus~\cite{Dershowitz_1991, Kennaway_1995, Kennaway_1997} for the
computer scientists. However, the generalised confluence property studied in
infinitary $\Sigma/\lambda$-term rewriting is different syntactically from the
$\delta$-confluence. We will call the former \emph{infinitary topological
confluence} property and the latter \emph{finitary topological confluence}.
From the general point of view of topological rewriting theory, it is shown
in~\cite{Chenavier_2025} that infinitary topological confluence is strictly
stronger than finitary topological confluence. However, the authors used
the above rewriting-theoretic characterisation of standard bases in terms
of $\delta$-confluence to show that infinitary and finitary topological
confluences are actually equivalent in commutative formal power series.

The goal of the present paper consists in giving a more rewriting-theoretic
approach to standard bases of commutative formal power series than before
presented in the mentioned literature. We will prove by new and more
elementary means the results of~\cite{Chenavier_2020} about the characterisation
of standard bases in terms of finitary topological confluence and
of~\cite{Chenavier_2025} about the equivalence between the notions of
generalised confluence. This will be part of the main contribution of this
article stated among the results of Theorem~\ref{thm:charac-standard-bases}.
That theorem states the properties that are analogous to the \Grobner{} basis
theory but for standard bases. We will also prove in
Theorem~\ref{thm:decide-ideal} that the equivalence relation generated by the
topological rewriting relation is exactly the congruence relation modulo the
ideal generated by the rewrite rules.

In Section~\ref{sct:preliminaries}, we will first introduce in
Subsection~\ref{ssct:complete-local-rings} the notions of local
rings, complete local rings, complete local equicharacteristic Noetherian rings,
then, in Subsection~\ref{ssct:formal-power-series} we will present monomial
orders, commutative formal power series, the Cohen structure theorem and some of
its consequences. In Section~\ref{sct:top-rew-metric-spaces}, we start by
recalling the basics of topological rewriting theory, then we give new results
regarding normal forms and some uniqueness properties and conclude
Subsection~\ref{ssct:topological-rewriting} with the definitions of infinitary
and finitary topological confluences. In
Subsection~\ref{ssct:attractivity-of-nf}, we study some notion of
attractivity of normal forms formalised in the context of topological rewriting
systems over a metric space. We follow by
Subsection~\ref{ssct:fin-top-conf-implies-unique-nf} in which we prove that,
under some previously defined assumptions, the finitary topological confluence
property implies the unique normal form property. Then, in
Subsection~\ref{ssct:rew-on-formal-power-series}, we remind the reader about the
reduction on formal power series that gives rise to the topological rewriting
system we want to study. In Subsection~\ref{ssct:standard-representations}, we
recall the original language in which standard bases were characterised. In
Subsection~\ref{ssct:cong-ideal}, we prove a few useful lemmas and the fact that
the congruence relation modulo the ideal generated by the rewrite rules is
exactly the equivalence relation generated by the topological rewriting relation
(with or without chains). Finally, we will conclude with the main contribution
of the paper: characterisations of standard bases in terms of rewriting theory.

\section{Preliminary notions}\label{sct:preliminaries}

\subsection{Complete local Noetherian rings}\label{ssct:complete-local-rings}

The basic object of interest for our purposes are the so-called \emph{local
rings}. Those are the basis of many modern theories on manifolds and schemes in
differential and algebraic geometries, as the main object of study in those
cases can be viewed as a \emph{locally-ringed space}, that is to say, a
topological space (with potentially additional properties) together with a sheaf
of rings (or algebras) such that the stalk above each point is a local ring.
Prime examples of those local rings as stalks are sets of germs of continuous,
differentiable, smooth, analytic, holomorphic, or rational functions on the base
space or topological manifold.

\begin{definition}[Local ring]\label{def:local-ring}
    Let $K$ be a commutative ring. If $K$ admits a unique maximal ideal, then
    $K$ is said to be a \emph{local ring}, where a maximal ideal is any ideal
    $I$ of $K$ such that we have $I \neq K$ and, for every ideal $J$ of $K$, if
    $I \subsetneq J$, then $J = K$.
\end{definition}

If $K$ is a local ring and $\mathfrak{m}$ is a maximal ideal of $K$, we write
$(K, \mathfrak{m})$ to express the fact that $K$ is a local ring with unique
maximal ideal $\mathfrak{m}$. By well-known basic results of commutative
algebra, if $(K, \mathfrak{m})$ is a local ring, then the quotient ring $\kappa
:= K/\mathfrak{m}$ is actually a field, since $\mathfrak{m}$ is maximal. We call
that field $\kappa$ the \emph{residue field} of the local ring $(K,
\mathfrak{m})$. The topology on rings we are concerned with in our context is
the \emph{$I$-adic topology} which is induced by an ideal $I$ in the ring at
study.

\begin{definition}[$I$-adic topology]\label{def:adic-topology}
    Let $K$ be a commutative ring and $I$ be one of its ideals. We define the
    \emph{$I$-adic topology on $K$} as the topology whose basis of open subsets
    is given by the sets:
    \[
        \ens{x + I^k}{x \in K, k \in \N}.
    \]
    Endowed with that $I$-adic topology, $K$ is a \emph{topological ring}.
\end{definition}

In general, this topology is not metrisable. However, a consequence of the
Artin-Rees lemma~\cite{Rees_1956} known as the \emph{Krull's Intersection
Theorem} allows us to describe a metric compatible with this topology when the
ring is commutative Noetherian as well as an integral domain or a local ring.

\begin{proposition}\label{prop:adic-topology-metric}
    Let $K$ be a commutative Noetherian ring which is an integral domain or a
    local ring. For any proper ideal $I$ of $K$, the following map $d_I$:
    \[
        \functionDefinition{d_I}{K \times K}{\RnonNeg}{(a, b)}{d_I(a, b) :=
        \dfrac{1}{2^{\sup{\ens{k \in \N}{a - b \in I^k}}}},}
    \]
    with the convention that, if the set $\ens{k \in \N}{a - b \in I^k}$ is not
    bounded above, then the supremum is equal to $\infty$ and, subsequently,
    $2^{-\infty} = 0$, defines a metric on $K$ which is compatible with the
    $I$-adic topology.
\end{proposition}

It is then natural to ask about the metric space completion of an $I$-adic
topology induced by a metric.

\begin{definition}[$I$-adic completion]\label{def:adic-completion}
    Let $K$ be a commutative Noetherian ring and $I$ an ideal of $K$ in such a
    way that $(K, d_I)$ is a metric space. Then, the \emph{$I$-adic completion}
    of $K$ is the metric space completion $(\hat{K}, \hat{d}_I)$ of $(K, d_I)$.
\end{definition}

\begin{remark}    
    Actually, the $I$-adic completion need not be defined only for topological
    rings which are also metric spaces but we restrict ourselves to that case
    here for simplicity.
\end{remark}

Since $K$ with the $I$-adic topology is a topological ring, by continuity of
operations, the $I$-adic completion $\hat{K}$ together with the topology induced
by the metric $\hat{d}_I$ is a \emph{complete topological ring} which contains
$K$ as a subring. Let us show that the $\hat{I}$-adic topology on $\hat{K}$, where $\hat{I}$
is the ideal generated by $I$ in $\hat{K}$ when $I$ is viewed as a subset of
$\hat{K}$, is the same as the topology induced by $\hat{d}_I$ on $\hat{K}$. 
First, recall that it can be shown by induction that, for any $r \in \N$, we
have $\hat{I}^r = \hat{K}I^r$. Now, it suffices to show that, for any point $x$
in any basic open subset $U$ in one of the topologies, there exists a basic open
subset $V$ in the other topology such that $x \in V \subseteq U$.

In one direction, let $x \in \hat{K}$ and $\rho \in \Rpos$. Consider $r$ the
smallest integer such that $2^{-r} < \rho$. Let $y \in x + \hat{I}^r$ and take
$(x_k)_{k \in \N}$ and $(y_k)_{k \in \N}$ two Cauchy sequences in $K$ that
represent $x$ and $y$ respectively in $\hat{K}$. Hence, there exists $\hat{a}
\in \hat{I}^r$ such that $\lim_{k \to \infty} x_k - y_k = \hat{a}$. But, by the
reminder above, it follows that there exist $\ell \in \N$,
$\DOTS{\hat{z}_1}{\hat{z}_\ell} \in \hat{K}$ and $\DOTS{a_1}{a_\ell} \in I^r$
such that $\hat{a} = \sum_{i = 1}^{\ell} \hat{z}_i a_i$. Consider, for each $i
\in \{\DOTS{1}{\ell}\}$, a Cauchy sequence $(z_{i,k})_{k \in \N}$ in $K$ that
represents $\hat{z}_i$ in $\hat{K}$. It follows that:
\[
  \lim_{k \to \infty} \sum_{i = 1}^{\ell} z_{i,k} a_i = \lim_{k \to \infty} x_k
  - y_k
\]
It is clear that, for any $k \in \N$, the set $\ens{j \in \N}{\sum_{i =
1}^{\ell} z_{i,k}a_i \in I^j}$ contains $r$ by construction. Hence, its supremum
is well-defined and is greater than or equal to $r$. 

Now, for $k \in \N$, we get:
\begin{align*}
  d_I(x_k, y_k) &= 2^{-\sup\ens{j \in \N}{x_k - y_k \in I^j}} = d_I(x_k - y_k, 0)
  \\
  d_I(x_k, y_k)
    &\leq d_I\left(x_k - y_k, \sum_{i = 1}^{\ell} z_{i,k}a_i\right) +
  d_I\left(\sum_{i = 1}^{\ell} z_{i,k}a_i, 0\right) &&\text{triangular
  inequality}\\
  d_I(x_k, y_k)
    &\leq d_I\left(x_k - y_k, \sum_{i = 1}^{\ell} z_{i,k}a_i\right) +
    2^{-\sup\ens{j \in \N}{\sum_{i = 1}^{\ell} z_{i,k}a_i \in I^j}} &&\text{by
    definition} \\
  d_I(x_k, y_k) 
    &\leq d_I\left(x_k - y_k, \sum_{i = 1}^{\ell} z_{i,k}a_i\right) + 2^{-r}
    &&\text{see above}
\end{align*}

Finally, by making $k$ tend to infinity, we get:
\begin{align*}
  \hat{d}_I(x, y)
    = \lim_{k \to \infty} d_I(x_k, y_k) &\leq \lim_{k \to \infty} d_I\left(x_k -
      y_k, \sum_{i = 1}^{\ell} z_{i,k}a_i\right) + 2^{-r} \\
  \hat{d}_I(x, y)
    &\leq 0 + 2^{-r} \\
  \hat{d}_I(x, y)
    &< \rho
\end{align*}

Conversely, let $x \in \hat{K}$ and $r \in \N$. Let us show that there exists
$\rho \in \Rpos$ such that every $y \in \hat{K}$ satisfying $\hat{d}_I(x, y) <
\rho$ is in $x + \hat{I}^r$. Consider $\rho := 2^{-r}$. Let $y \in
\hat{K}$ be such that $\hat{d}_I(x, y) < \rho$. Take $(x_k)_{k \in \N}$ and
$(y_k)_{k \in \N}$ two Cauchy sequences in $K$ representing $x$ and $y$
respectively in $\hat{K}$. It follows that there exists $K_\rho \in \N$ such
that, for all $k \geq K_\rho$, we have $\hat{d}_I(x_k, y_k) = d_I(x_k, y_k) <
\rho = 2^{-r}$, hence, $\sup\ens{j \in \N}{x_k - y_k \in I^j} > r$. This means
that, for any $k \geq K_\rho$, $r \in \ens{j \in \N}{x_k - y_k \in I^j}$ and,
thus, $x_k - y_k \in I^r$. Making $k$ tend towards infinity, we obtain $x - y
\in \hat{I}^r$, thus $y \in x + \hat{I}^r$, as desired.

To understand the following definition, recall that, by
Proposition~\ref{prop:adic-topology-metric}, the $\mathfrak{m}$-adic topology on
a local ring $(K, \mathfrak{m})$ is metrisable by the metric $d_{\mathfrak{m}}$
defined in that proposition.
\begin{definition}[Complete local Noetherian
    ring]\label{def:complete-local-ring}
    Let $(K, \mathfrak{m})$ be a local ring where $K$ is Noetherian. If $(K,
    d_{\mathfrak{m}})$ is complete as a metric space, then we say that $(K,
    \mathfrak{m})$ is a \emph{complete local Noetherian ring}.
\end{definition}

As a complete local Noetherian ring is not necessarily an integral domain (and
thus not a field), the characteristic of the ring can be something else than $0$
or a prime number. It turns out that it can also be a power of a prime number.
However, we will dismiss that possibility by considering only complete local
Noetherian rings that have the same characteristic as their residue field, which
is thus necessarily $0$ or prime.

\begin{definition}[Equicharacteristic]\label{def:equicharacteristic}
    Let $(K, \mathfrak{m})$ be a local ring and $\kappa := K/\mathfrak{m}$ its
    residue field. If the characteristic of $K$ is the same as the
    characteristic of $\kappa$, then $(K, \mathfrak{m})$ is said to be
    \emph{equicharacteristic}.
\end{definition}

\begin{remark}
    For any local ring $(K, \mathfrak{m})$, being equicharacteristic is
    equivalent to containing a subring which is actually a field. Some authors
    prefer to use that other statement and might not even use the term
    equicharacteristic (for instance~\cite{Bourbaki_2006}).
\end{remark}

It is a known result from~\cite{Zariski_1975} that, for an ideal $I$ in a
commutative ring $K$, the ideals of $K$ are topologically closed in the $I$-adic
topology if, and only if, $I$ is contained in the so-called \emph{Jacobson
radical}, which in our context is just the intersection of all maximal ideals of
$K$. In particular, we have the following proposition.
\begin{proposition}\label{prop:local-ring-is-zariski}
    Any ideal in a local ring $(K, \mathfrak{m})$ where $K$ is Noetherian is
    topologically closed for the $\mathfrak{m}$-adic topology.
\end{proposition}

\subsection{Formal power series and standard
bases}\label{ssct:formal-power-series}

Throughout this subsection, $n$ is a fixed positive integer, $\X$ are distinct
indeterminates and $\K$ a (commutative) field of arbitrary characteristic.
Denote by $[\X]$ (or $[\x]$ for short) the free commutative monoid generated by
$\{\X\}$ (whose law of composition is denoted multiplicatively by $\cdot$), call
elements of $[\x]$ \emph{monomials} and denote by $1$ the \emph{empty monomial},
\ie{} the identity element of $[\x]$. Elements of $[\x]$ are of the form $\x^\mu
:= \DOTS[]{x_1^{\mu_1}}{x_n^{\mu_n}}$ where $\mu := (\DOTS{\mu_1}{\mu_n}) \in
\N^n$. In particular, we have the identity $1 = \x^{0_n} =
\DOTS[]{x_1^0}{x_n^0}$. Define the \emph{total degree} of a monomial $m :=
\x^\mu \in [\x]$ by the formula $\deg(m) := \abs{\mu} :=
\DOTS[+]{\mu_1}{\mu_n}$.

\begin{definition}[Monomial order]\label{def:monomial-order}
    We call \emph{monomial order on $[\x]$} any total order $\leq$ on $[\x]$
    such that, for all $m, m_1, m_2 \in [\x]$, if we have $m_1 \leq m_2$, then
    $m \cdot m_1 \leq m \cdot m_2$.
\end{definition}

As any partial order, a monomial order $\leq$ is uniquely determined by its
associated strict order $<$. Therefore, we will use the term ``monomial order''
for $\leq$ and $<$ interchangeably. Note how, for any monomial order $<$ on
$[\x]$, the opposite order $\op{<}$ is also a monomial order on $[\x]$.  A
monomial order $<$ on $[\x]$ is said to be \emph{admissible} if, for all $m \in
[\x]$, $1 \leq m$. This is equivalent to saying that $\leq$ is a well-order on
$[\x]$. Some authors~\cite{Greuel_2002} use the term \emph{global order} for
admissible monomial order and call the opposite order of a global order a
\emph{local order}. We say that a monomial order $<$ is \emph{compatible with
the degree} if, for all $m_1, m_2 \in [\x]$ such that $\deg(m_1)$ is strictly
smaller than $\deg(m_2)$, then necessarily $m_1 < m_2$. An example of an
admissible monomial order compatible with the degree is the
``degree-lexicographic order'' (or ``deglex'' for short). An admissible monomial
order compatible with the degree is of order type $\omega$, where $\omega$
denotes as usual the first infinite ordinal. 

Denote by $\KX$ ($\Kx$ for short) the \emph{$\K$-algebra of $n$-multivariate
polynomials with coefficients in $\K$}. By construction, any polynomial $f \in
\Kx$ is a finite linear combination of monomials in $[\x]$ with coefficients in
$\K$ and, thus, it makes sense to define, for any monomial $m \in [\x]$, the
scalar $\coeff{f}{m}$ called the \emph{coefficient in $f$ associated to $m$}. It
follows that the set $\supp{f} := \ens{m \in [\x]}{\coeff{f}{m} \neq 0}$, called
the \emph{support of the polynomial $f$} is finite and that any polynomial $f$
can thus be represented by the finite linear combination $f = \sum_{m \in
\supp{f}} \coeff{f}{m} m$. The set $\K$ can be embedded into $\Kx$ by the
morphism of algebras $\K \ni \lambda \mapsto \lambda 1 \in \Kx$ where, here, $1$
is the empty monomial. The image of $\K$ by that morphism is exactly the set of
\emph{constant polynomials} and it is the complementary subspace in
$\Kx$ to the subspace given by the ideal $I(\DOTS{x_1}{x_n})$ in $\Kx$ generated
by $\{\DOTS{x_1}{x_n}\}$. Denote that ideal by $I(\x)$ for short and note how
it follows that $I(\x)$ is a proper ideal of $\Kx$ since $1$ is a non-zero
constant polynomial. By Hilbert's Basis Theorem, $\Kx$ is Noetherian and, since
we assume $\K$ is a field, also an integral domain, we can conclude that the $I(\x)$-adic
topology on the ring $\Kx$ is metrisable. Denote the associated metric by
$\delta := d_{I(\x)}$ defined in Proposition~\ref{prop:adic-topology-metric}.

We define the topological ring of (commutative) \emph{formal power series in $n$
indeterminates with coefficients in $\K$} as the $I(\x)$-adic completion (see
Definition~\ref{def:adic-completion}) of $(\Kx, \delta)$ and we denote it
$(\KXX, \delta)$ (or $(\Kxx, \delta)$ for short). Since $\K$ is embedded into
the ring $\Kx$ and $\Kx$ into $\Kxx$, it follows that $\Kxx$ is actually a
topological $\K$-algebra by defining the scalar multiplication as the formal
power series multiplication restricted to the subring identified with $\K$. It
is a well-known result that this definition of $\Kxx$ coincides with the
so-called ``large algebra associated to the monoid $[\x]$'' of
\textsc{Bourbaki}~\cite{Bourbaki_2007}: $\Kxx$ is thus defined as the
set of functions from $[\x]$ to $\K$, including those whose support is infinite.
In other words, an element $f$ of $\Kxx$ can be represented as a, possibly
infinite, linear combination of monomials in $[\x]$ with coefficients in $\K$ by
formally writing it as the sum of every monomial $m$ in $[\x]$ preceded by its
associated coefficient, that is to say, the scalar $f(m)$:
\[
    f =: \sum_{m \in [\x]} f(m) m.
\]
Extending the notations for polynomials, we set $\coeff{f}{m} := f(m)$ the
coefficient associated to a monomial $m \in [\x]$ in a formal power series $f
\in \Kxx$ as given by that linear combination representation, as well as
$\supp{f} := \ens{m \in [\x]}{\coeff{f}{m} \neq 0}$ for a formal power series $f
\in \Kxx$. 

Now, since $\delta$ induces the $I(\x)$-adic topology on $\Kxx$ when
$I(\x)$ is viewed as an ideal of $\Kxx$, we can give a somewhat more convenient
definition of $\delta$:
\begin{equation}\label{eq:delta}
    \forall f, g \in \Kxx,\quad \delta(f, g) = \frac{1}{2^{\val{f - g}}},
\end{equation}
where $\val{h}$ is equal to the smallest degree of monomials in $\supp{h}$, for
any non-zero formal power series $h \in \Kxx \setminus \{0\}$ or $\infty$ if $h
= 0$.

Let us prove the following property.
\begin{proposition}\label{prop:Kxx-not-finitely-generated}
    The algebra $\Kxx$ is not finitely generated as a $\K$-algebra.
\end{proposition}
\begin{proof}
  First of all, recall that a commutative ring $K$ is said to be
  \emph{Jacobson} when every prime ideal is the intersection of maximal ideals.
  If $\K$ is a field, then it is straightforwardly Jacobson. By a general form
  of Hilbert's \emph{Nullstellensatz}~\cite{Goldman_1951}, if $K$ is a Jacobson
  ring then any finitely generated $K$-algebra is also Jacobson. Let $\K$ be a
  field. Then any finitely generated $\K$-algebra $K$ which turns out to be a
  local ring is of Krull dimension $0$: indeed, let $\mathfrak{m}$ be the unique
  maximal ideal of $K$, then, since $K$ is Jacobson, the only prime ideal is
  $\mathfrak{m}$. Now, if furthermore $K$ is an integral domain, that is to say,
  $\{0\}$ is a prime ideal, then it is once again clear that the latter is the
  unique maximal ideal, and thus $K$ is a field. Hence, by contrapositive, let
  $K$ be a $\K$-algebra which is local as a ring as well as an integral domain,
  then, if $K$ is not a field, it cannot be finitely generated as a
  $\K$-algebra.  Finally, it is known that $\Kxx$ is a $\K$-algebra and is local
  as a ring as well as an integral domain but not a field.
\end{proof}

Now, let us recall the \emph{Cohen Structure Theorem}, originally proved by
\textsc{Cohen} in~\cite{Cohen_1946}. For our purposes on complete local
equicharacteristic Noetherian rings, we will reformulate the theorem
from~\cite[AC IX.30, Paragraphe 4, Num\'ero 3, Th\'eor\`eme 2]{Bourbaki_2006}:
\begin{theorem}\label{thm:cohen-structure}
    Let $(K, \mathfrak{m})$ be a complete local equicharacteristic Noetherian
    ring. Let us denote by $\kappa := K/\mathfrak{m}$ its residue field and
    by $d$ the Krull dimension of $K$.
    \begin{enumerate}
        \item\label{item:field-of-representatives}
            The ring $K$ contains a subring $\K$ which maps isomorphically onto
            $\kappa$ through the natural projection $K \longrightarrow
            K/\mathfrak{m} = \kappa$ and, therefore, is a field of same
            characteristic as both $\kappa$ and $K$. Thus, $K$ is a $\K$-algebra.
            Such a field $\K$ is a called \emph{field of representatives} for
            $(K, \mathfrak{m})$ (or, as originally called by \textsc{Cohen},
            \emph{coefficient field}).
        \item\label{item:number-of-variables}
            The quotient $\mathfrak{m}/\mathfrak{m}^2$ is a vector space over
            $\K$. Let $m$ be its $\K$-dimension.
        \item\label{item:Kxx-free-algebra}
            There exists an ideal $I$ of $\K[[\DOTS{x_1}{x_m}]]$ such that $K$
            and $\K[[\DOTS{x_1}{x_m}]]/I$ are isomorphic as $\K$-algebras.
        \item\label{item:sub-algebra}
            There exists a sub-$\K$-algebra $K_0$ of $K$ such that $K_0$ is
            isomorphic to $\K[[\DOTS{x_1}{x_d}]]$ and $K$ is a finitely
            generated $K_0$-module.
        \item\label{item:regular}
            If $(K, \mathfrak{m})$ is \emph{regular}, \ie, if $d =
            m$, then the $\K$-algebras $K$ and $\K[[\DOTS{x_1}{x_d}]]$ are
            isomorphic.
    \end{enumerate} 
\end{theorem}

Let us now introduce standard bases of formal power series. In order
to accomplish that, we need to define the notions of leading monomials and
leading coefficients.
\begin{definition}\label{def:leading}
    Let $<$ be an admissible monomial order on $[\x]$ and $f \in \Kxx \setminus
    \{0\}$. Define the \emph{leading monomial} of $f$ for $<$ as follows:
    \[
        \LM{<}{f} := \min_{<} \supp{f} \in [\x].
    \]
    We then define the \emph{leading coefficient} of $f$ for $<$ as $\LC{<}{f}
    := \coeff{f}{\LM{<}{f}} \in \K$.
\end{definition}

\begin{definition}[Standard basis]\label{def:standard-basis}
    Let $I$ be an ideal in $\Kxx$ and $<$ be an admissible monomial order. A
    \emph{standard basis of $I$ for $<$} is any set $G \subseteq I \setminus
    \{0\}$ such that, for all $f \in I \setminus \{0\}$,
    there exist $g \in G$ and $m \in [\x]$ that satisfy:
    \[
        \LM{\op{<}}{f} = m \cdot \LM{\op{<}}{g}.
    \]
\end{definition}
In the commutative setting that we study here, there always exists a
\emph{finite} standard basis for any ideal of $\Kxx$ and any admissible monomial
order.

It is known since \textsc{Hironaka}~\cite{Hironaka_1964, Hironaka_1964a} that, for a
fixed ideal $I$ in $\Kxx$ and a fixed monomial order $<$, for any $f
\in \Kxx$, there exists a unique formal power series, let us denote it here by
$r_f \in \Kxx$, such that $f \equiv r_f \mod I$ and no monomial in $\supp{r_f}$
is a multiple of a leading monomial for $<$ of a formal power series in
$I$. That formal power series $r_f$ is called the \emph{Hironaka remainder} of
$f$ for $<$ modulo $I$. It will serves as a canonical representative of the
equivalence class of $f$ modulo $I$.  The purpose of standard bases was
primarily to compute effectively that Hironaka remainder. In
Subsection~\ref{ssct:cong-ideal}, we will show that Hironaka
remainders are actually just the normal forms of a topological rewriting system
which satisfies a certain notion of confluence.

Here is a useful proposition which relates the degree of the leading
monomial of the difference of two formal power series for a certain monomial
order with the distance between the two formal power series.
\begin{proposition}\label{prop:deg-lm-val}
    Let $<$ be an admissible monomial order on $[\x]$ which is compatible with
    the degree. Then, for all $f, g \in \Kxx$ such that $f \neq g$, we have:
    \[
        \delta(f, g) = \dfrac{1}{2^{\deg\left(\LM{\op{<}}{f - g}\right)}}.
    \]
\end{proposition}
\begin{proof}
    Indeed, let $h \in \Kxx \setminus \{ 0 \}$, since the monomial order $<$ is
    admissible, the leading monomial $\LM{\op{<}}{h}$ is exactly the least (for
    $<$) monomial in the support $\supp{h}$. Now, since the order $<$ is
    compatible with the degree, \ie{} the degree function on monomials is
    non-decreasing, it follows that $\deg\left(\LM{\op{<}}{h}\right)$ is minimal
    among all the $\deg(m)$ with $m \in \supp{h}$. This exactly means that
    $\deg\left(\LM{\op{<}}{h}\right) = \val{h}$, hence the result.
\end{proof}

To end this subsection, we will prove
Proposition~\ref{prop:comp-loc-equi-noeth-ring-not-finitely-generated}. For that
purpose, we will require the following lemma:
\begin{lemma}[Artin-Tate~\cite{Artin_1951}]\label{lem:Artin-Tate}
    Let $K$ be a commutative Noetherian ring and let $A_0 \subseteq A$ be two
    commutative $K$-algebras. If $A$ is finitely generated as a $K$-algebra
    and if $A$ is a finitely generated $A_0$-module, then $A_0$ is finitely
    generated as a $K$-algebra.
\end{lemma}

Now, using previous results we prove the following proposition.
\begin{proposition}\label{prop:comp-loc-equi-noeth-ring-not-finitely-generated}
    Let $(K, \mathfrak{m})$ be a complete local equicharacteristic Noetherian
    ring which has a positive Krull dimension. Let $\kappa :=
    K/\mathfrak{m}$ be its residue field. Then, $K$ is not finitely generated as a
    $\kappa$-algebra, where the action of $\kappa$ is defined via
    Statement~\ref{item:field-of-representatives} of
    Theorem~\ref{thm:cohen-structure}.
\end{proposition}
\begin{proof}
    By Theorem~\ref{thm:cohen-structure} statement
    number~\ref{item:sub-algebra}, there exists a sub-$\kappa$-algebra $K_0$ of
    $K$ isomorphic to a $\kappa$-algebra of formal power series in
    finitely many variables in such a way that $K$ is actually a finitely
    generated $K_0$-module.  However, by
    Proposition~\ref{prop:Kxx-not-finitely-generated}, $K_0$ cannot be finitely
    generated as a $\kappa$-algebra. Hence, by contrapositive of
    Lemma~\ref{lem:Artin-Tate}, we conclude that $K$ cannot be finitely
    generated as a $\kappa$-algebra.
\end{proof}

A consequence of
Proposition~\ref{prop:comp-loc-equi-noeth-ring-not-finitely-generated} is that,
if we were to use \Grobner{} bases methods to compute in a complete local
equicharacteristic Noetherian ring, we would have to work on a polynomial ring
with infinitely many variables, for which many essential results of \Grobner{}
bases in finitely many variables fail to translate, especially when it comes to
effectiveness. This is why standard bases are a useful alternative tool.

It follows from Proposition~\ref{prop:local-ring-is-zariski} and the fact that
$I(\x)$ is the unique maximal ideal of the local ring $\Kxx$ with which we
defined the topology that any ideal in $\Kxx$ is topologically closed for the
topology induced by the metric $\delta$.

\section{Topological rewriting}\label{sct:top-rew-metric-spaces}

\subsection{Basic notions and results}\label{ssct:topological-rewriting}

In this subsection, we introduce the basics of \emph{topological rewriting
theory} as originally presented in~\cite{Chenavier_2020}. It is a generalisation
of classical rewriting theory in which we study the reduction relation with
respect to a topology on the base set.

\begin{definition}[Topological rewriting
    system]\label{def:topological-rewriting-system}
    An \emph{(abstract) topological rewriting system} is the data of $(X, \tau,
    \rew)$ where $(X, \tau)$ is a topological space and $\rew$ is a binary
    relation on $X$ called the \emph{base rewriting} or \emph{one-step
    reduction} relation.
\end{definition}

As for the classical setting, we study the system not only through the base
relation but also via super-relations of the base relation. For instance, in the
classical context, we take interest in the transitive reflexive closure of the
base relation, which we denote by $\transReflRew$. This relation describes a
finite sequence of elementary computations. In topological rewriting
systems, we investigate a more general kind of relation, namely the
\emph{topological rewriting relation} as defined thereafter.

\begin{definition}\label{def:topological-rewriting-relation}
    Let $(X, \tau, \rew)$ be a topological rewriting system. The
    \emph{topological rewriting relation} of that system, denoted $\topRew$, is
    defined as the topological closure of the relation $\transReflRew$ viewed as
    a subspace of $X \times X$ equipped with the product topology $\dTop{X}
    \times \tau$, where $\dTop{X}$ is the discrete topology on $X$.
\end{definition}

Before noticing that this is indeed a generalisation of classical abstract
rewriting systems, we will characterise in the following proposition that two
elements $a, b \in X$ are related through that topological rewriting relation as
$a \topRew b$ if, and only if, $a$ rewrites in finitely many steps arbitrarily close
to $b$ (in the sense of the topology $\tau$), \ie{} there are finite sequences
of elementary computations starting at $a$ that come arbitrarily close to $b$.

\begin{proposition}\label{prop:top-rew-charac}
    Let $(X, \tau, \rew)$ be a topological rewriting system. Then, for all $a,
    b \in X$, we have $a \topRew b$ if, and only if, for all neighbourhoods $U$
    of $b$ in $(X, \tau)$, there exists $c \in U$ such that $a \transReflRew c$.
\end{proposition}
\begin{proof}
    Let $\mathfrak{X}^2$ be the topological space $X \times X$ endowed with the
    product topology $\dTop{X} \times \tau$.

    Suppose $(a, b)$ is in the closure of $\transReflRew$ when viewed as a
    subset of $\mathfrak{X}^2$. Let $U$ be a neighbourhood of $b$ in $(X,
    \tau)$. Then $\{a\} \times U$ is a neighbourhood of $(a, b)$ in
    $\mathfrak{X}^2$ and therefore, there exists $(a, c) \in \{a\} \times U$
    such that $a \transReflRew c$.

    Conversely, assume that for all neighbourhoods $U$ of $b$ in $(X, \tau)$ we
    have $c \in U$ such that it satisfies $a \transReflRew c$. But, by
    definition of the product space, for any neighbourhood $V$ of $(a, b)$ in
    $\mathfrak{X}^2$, there exist a neighbourhood $W$ of $a$ in $\left(X,
    \dTop{X}\right)$ and a neighbourhood $U$ of $b$ in $(X, \tau)$ such that
    they satisfy the relation $W \times U \subseteq V$. Now, by assumption,
    there exists $c \in U$ such that $a \transReflRew c$ which means exactly
    that $(a, b)$ is in the topological closure of $\transReflRew$ when viewed
    as a subset of $\mathfrak{X}^2$, since $a \in W$.
\end{proof}

\begin{remark}\label{rmk:generalisation-of-ars}
    Note how $\transReflRew$ is always a subrelation of $\topRew$. Now,
    conversely, notice that, when we consider $\tau := \dTop{X}$, it follows
    from Proposition~\ref{prop:top-rew-charac} that the relation $\topRew$
    associated to the topological rewriting system $\left(X, \dTop{X},
    \rew\right)$ is exactly equal to $\transReflRew$ as binary relations. In
    that sense, the topological rewriting systems are generalisations of
    abstract rewriting systems because most of the subsequent definitions in
    topological rewriting theory turn out to translate back to well-known
    properties of classical rewriting theory when considering the discrete
    topology.
\end{remark}

The notion of normal forms remains the same as in the classical setting.

\begin{definition}[Normal form]\label{def:normal-form}
    Let $(X, \tau, \rew)$ be a topological rewriting system. Any $a \in X$ is
    called a \emph{normal form} of the system if $\ens{b \in X}{a \rew b} =
    \emptyset$.
\end{definition}

It is possible to characterise normal forms in terms of the topological
rewriting relation when we allow ourselves to assume a property on the
underlying topological space. Recall that a $T_1$-space is a topological space
$(X, \tau)$ such that for all $x \in X$, the intersection of all neighbourhoods
of $x$ in $(X, \tau)$ is exactly equal to the singleton $\{x\}$.

\begin{lemma}\label{lem:charac-normal-forms}
    Let $(X, \tau, \rew)$ be a topological rewriting system.
    \begin{enumerate}
        \item\label{item:normal-form-T_1-space}
            Assume that $(X, \tau)$ is a $T_1$-space. Then, for any normal
            form $a$ of the system and for all $b \in X$, if $a \topRew b$ then
            $b = a$.
        \item\label{item:normal-form-anti-reflexive}
            Assume that $\rew$ is anti-reflexive and let $a \in X$. If, for
            all $b \in X$ such that $a \topRew b$ we have $b = a$, then $a$ is a
            normal form for the system.
    \end{enumerate}
\end{lemma}
\begin{proof}
    For statement number~\ref{item:normal-form-T_1-space}, assume $(X, \tau)$ is
    a $T_1$-space and let $a$ be a normal form and consider $b \in X$ such that $a
    \topRew b$. Then, by Proposition~\ref{prop:top-rew-charac}, for all
    neighbourhoods $U$ of $b$ for the topology $\tau$ there exists $c \in U$
    such that $a \transReflRew c$. Now, since $a$ is a normal form, $c = a$.
    In other words, $a$ is contained in all neighbourhoods of $b$. Finally,
    since $(X, \tau)$ is a $T_1$-space by assumption, it follows that $b = a$.

    For statement number~\ref{item:normal-form-anti-reflexive}, assume $\rew$
    anti-reflexive and let $a \in X$ be an element such that $\ens{b \in X}{a
    \topRew b} = \{a\}$. By contradiction, assume there exists $b \in X$ such
    that $a \rew b$. Since $\rew$ is a subrelation of $\topRew$, we have $a
    \topRew b$. By assumption, it follows that $b = a$, hence $a \rew a$, which
    contradicts the anti-reflexivity.
\end{proof}

One of the goals of topological rewriting theory is to study systems where
classical confluence fails to be strictly verified, as well as systems which are
not \emph{terminating} (\textit{a.k.a.} \emph{strongly normalising}). However,
it is still useful to consider systems in which a normal form can be reached
through a super-relation of $\rew$, starting at any element of the base space.
This is the purpose of the following definition.

\begin{definition}\label{def:normalising}
    Let $(X, \tau, \rew)$ be a topological rewriting system and let $\squigRew$
    be a subrelation of $\topRew$. Denote by $\squigTransReflRew$ the transitive
    reflexive closure of $\squigRew$. We say that the system is
    \emph{normalising with respect to $\squigRew$} if, for all $a \in X$ there
    exists a normal form $b$ of the system such that $a \squigTransReflRew b$,
    in which case we say ``$b$ is a normal form of $a$ reached by $\squigRew$''.
\end{definition}

Rewriting theory is especially concerned with the existence and uniqueness of
normal forms of arbitrary elements of the base space as reached by a certain
super-relation of the base rewriting relation. Definition~\ref{def:normalising}
gives the existence and we now give three definitions related to the uniqueness
of normal forms.

\begin{definition}\label{def:nf-props}
    Let $(X, \tau, \rew)$ be a topological rewriting system and let $\squigRew$
    be a subrelation of $\topRew$. Let $\squigTransReflRew$ be the transitive
    reflexive closure of $\squigRew$ and $\squigEquivRew$ the equivalence
    relation generated by $\squigRew$. We say that the system:
    \begin{enumerate}
        \item\label{item:nf-prop}
            has the \emph{normal form property with respect to $\squigRew$}
            if, for all $a \in X$ and any normal form $b$ of the system such
            that $a \squigEquivRew b$, then $a \squigTransReflRew b$.
        \item\label{item:un-prop}
            has the \emph{unique normal form property with respect to
            $\squigRew$} if, for all normal forms $a$ and $b$ of the system such
            that $a \squigEquivRew b$, then $a = b$.
        \item\label{item:unique-nf}
            has \emph{unique normal forms reached by $\squigRew$} if,
            for all $a \in X$ and any normal forms $b$ and $c$ of the system
            such that $a \squigTransReflRew b$ and $a \squigTransReflRew c$,
            then $b = c$.
    \end{enumerate}
\end{definition}

It turns out that, under a specific normalisation assumption and over a
$T_1$-space, these three properties are equivalent. This is the content of the
following proposition.

\begin{proposition}\label{prop:equivalence-of-nf-properties}
    Let $(X, \tau, \rew)$ be a topological rewriting system and let $\squigRew$
    be a subrelation of $\topRew$. Consider all the definitions with respect to
    $\squigRew$. We then have:
    \begin{enumerate}
        \item\label{item:un-unique}
            The unique normal form property implies that the system has unique
            normal forms reached by $\squigRew$.
        \item\label{item:nf-un}
            If $(X, \tau)$ is a $T_1$-space and if $\topRew$ is transitive, then
            the normal form property implies the unique normal form property.
        \item\label{item:un-nf}
            If the system is normalising, then the unique normal form property
            implies the normal form property.
        \item\label{item:unique-un}
            If $(X, \tau)$ is a $T_1$-space and if the system is normalising,
            then, when the system has unique normal forms reached by
            $\squigRew$, it has the unique normal form property.
    \end{enumerate}
\end{proposition}
\begin{proof}
    Statements~\ref{item:un-unique} and~\ref{item:nf-un} are straightforward
    with what preceeds.

    For statement~\ref{item:un-nf}, let $a \in X$ and $b$ be a normal form of the
    system such that $a \squigEquivRew b$. By normalisation hypothesis, there
    exists a normal form $c$ of the system such that $a \squigTransReflRew c$.
    Hence, it follows that $c \squigEquivRew b$ where $c$ and $b$ are normal
    forms. Hence, by assumption, it follows that we have $b = c$ and, thus, $a
    \squigTransReflRew b$, the desired result.

    For statement~\ref{item:unique-un}, let $a$ and $b$ be normal forms of the
    system such that $a \squigEquivRew b$. Denote by $\squigInvRew$ and
    $\squigTransReflInvRew$ the opposite relations of $\squigRew$ and
    $\squigTransReflRew$ respectively. If we let $\squigSymRew$ be the
    symmetric closure of $\squigRew$, the assertion $a \squigEquivRew b$ exactly
    means that there exists a finite tuple of elements $(\DOTS{c_1}{c_\ell}) \in
    X^\ell$, for $\ell \in \N$, such that:
    \[
        a \squigSymRew c_1 \squigSymRew c_2 \squigSymRew \cdots \squigSymRew
        c_\ell \squigSymRew b.
    \] 

    Let us call a ``valley'' in a sequence $c_1 \squigSymRew c_2 \squigSymRew
    \cdots \squigSymRew c_\ell$ any index $i \in \{\DOTS{2}{\ell-1}\}$ such that
    $c_{i-1} \squigRew c_i \squigInvRew c_{i+1}$. 

    Let us show that, for any $\ell \in \N$ and any tuple $(\DOTS{c_1}{c_\ell})
    \in X^\ell$ such that:
    \[
        a \squigInvRew c_1 \squigSymRew c_2 \squigSymRew \cdots \squigSymRew
        c_\ell \squigRew b.
    \]
    which contains $\nu \in \N \setminus \{0\}$ valleys, there exists $\ell' \in
    \N$ and $(\DOTS{d_1}{d_{\ell'}}) \in X^{\ell'}$ such that:
    \[
        a \squigInvRew d_1 \squigSymRew d_2 \squigSymRew \cdots \squigSymRew
        d_{\ell'} \squigRew b.
    \]
    which contains $\nu-1$ valleys.

    Indeed, if there are $\nu \geq 1$ valleys in the first sequence, there is a
    right-most one, by which we mean, the biggest index $i \in \{\DOTS{2}{\ell -
    1}\}$ which is a valley. In other words, we have the following diagram:
    \[ 
        c_{i-1} \squigRew c_i \squigInvRew c_{i+1} \squigInvRew c_{i+2}
        \squigInvRew \cdots \squigInvRew c_j \squigRew c_{j+1} \squigRew \cdots
        \squigRew c_\ell \squigRew b,
    \]
    for a unique $j \in \{\DOTS{i+2}{\ell}\}$.

    By normalisation hypothesis, there exists a normal form $d$ for the system
    that satisfies the relation $c_i \squigTransReflRew d$. Note how we are thus
    in the following situation:
    \[
        c_{i-1} \squigTransReflRew d \squigTransReflInvRew c_j
        \squigTransReflRew b.
    \]

    By assumption of unique normal forms reached by $\squigRew$, it follows that
    $b = d$. Thus, there exist $\ell'' \in \N$ and $(d_{i+k})_{0 \leq k \leq
    \ell''}$ such that $c_{i-1} \squigRew d_{i} \squigRew d_{i+1} \squigRew
    \cdots \squigRew d_{i+\ell''} \squigRew d = b$.

    It then suffices to define $\ell' := i + \ell''$ and $d_k := c_k$ for all $k
    \in \{\DOTS{1}{i-1}\}$ and it follows that the sequence:
    \[
        a \squigInvRew d_1 \squigSymRew d_2 \squigSymRew \cdots \squigSymRew
        d_{i-1} \squigRew d_i \squigRew d_{i+1} \squigRew \cdots \squigRew d_{i
        + \ell''} = d_{\ell'} \squigRew b.
    \]
    contains $\nu-1$ valleys.

    Now, since $\squigRew$ is a subrelation of $\topRew$, for all $c \in X$ such
    that $a \squigSymRew c$, we have:
    \begin{equation}\label{eq:a-topRew-c_1}
        a \topRew c \quad\text{if}\quad c \squigRew a \,\text{is false}.
    \end{equation}
    However, by hypothesis, the space $(X, \tau)$ is a $T_1$-space, and thus, by
    Lemma~\ref{lem:charac-normal-forms} and since $a$ is a normal form, it
    follows that (\ref{eq:a-topRew-c_1}) yields the following:
    \[
        c = a \quad\text{if}\quad c \squigRew a \,\text{is false}.
    \]
    Similarly, for any $c \in X$ such that $c \squigSymRew b$, we have $c = b$
    if $c \squigRew b$ is false. Utilising the contrapositive, this means that,
    for any sequence and $\ell \geq 1$:
    \[
        a \squigSymRew c_1 \squigSymRew c_2 \squigSymRew \cdots \squigSymRew
        c_\ell \squigSymRew b.
    \]
    there exist $i_0, i_1 \in \{\DOTS{1}{\ell}\}$ such that $i_0 \leq i_1$:
    \[
      a \squigInvRew c_{i_0} \squigSymRew c_{i_0+1} \squigSymRew c_{i_0+2}
      \squigSymRew \cdots \squigSymRew c_{i_1 - 1} \squigSymRew c_{i_1}
      \squigRew b
    \]

    Therefore, since $a \squigEquivRew b$, there necessarily exist $\ell \in \N$
    and a tuple $(\DOTS{c_1}{c_\ell}) \in X^\ell$ such that:
    \[
        a \squigInvRew c_1 \squigSymRew c_2 \squigSymRew \cdots \squigSymRew
        c_\ell \squigRew b.
    \]
    which, morevoer, contains no valley, thanks to the process given earlier.
    This means that there exists $i \in
    \{\DOTS{1}{\ell}\}$ verifying:
    \[
        a \squigInvRew c_1 \squigInvRew c_2 \squigInvRew \cdots \squigInvRew c_i
        \squigRew c_{i+1} \squigRew \cdots \squigRew c_\ell \squigRew b,
    \]
    At this point, it suffices to use the assumption of
    Definition~\ref{def:nf-props}-\ref{item:unique-nf} to conclude that $a = b$
    since the above sequence implies $c_i \squigTransReflRew a$ and $c_i
    \squigTransReflRew b$.    
\end{proof}

We now introduce two different notions of confluence for topological rewriting
systems, which are exactly the classical confluence when considering the
discrete topology.

\begin{definition}[Topological confluence]\label{def:top-conf}
    Let $(X, \tau, \rew)$ be a topological rewriting system. We say that the
    system is:
    \begin{enumerate}
        \item \emph{finitary topologically confluent} if, for all $a, b, c \in
            X$ such that $a \transReflRew b$ and $a \transReflRew c$, then there
            exists $d \in X$ such that $b \topRew d$ and $c \topRew d$. In
            diagrams:
            
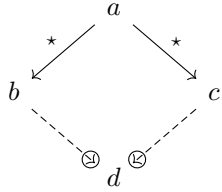
\[
                \begin{tikzcd}
                    & \ar{dl}[']{\star} a \ar{dr}{\star} & \\
                    b \ar[circle, dashed]{dr} & & \ar[circle, dashed]{dl} c \\
                                      & d &
                \end{tikzcd}
            \]
        \item \emph{infinitary topologically confluent} if, for all $a, b, c \in
            X$ such that $a \topRew b$ and $a \topRew c$, then there
            exists $d \in X$ such that $b \topRew d$ and $c \topRew d$. In
            diagrams:
            
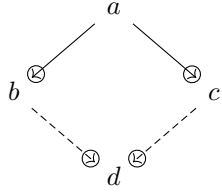
\[
                \begin{tikzcd}
                    & \ar[circle]{dl} a \ar[circle]{dr} & \\
                    b \ar[circle, dashed]{dr} & & \ar[circle, dashed]{dl} c \\
                                      & d &
                \end{tikzcd}
            \]
    \end{enumerate}
\end{definition}

\begin{remark}\label{rmk:infty-implies-finitary} 
    Since $\transReflRew$ is a subrelation of $\topRew$, it is clear that
    infinitary topological confluence implies finitary topological confluence of
    the system. See~\cite[Examples 4.1.3, 4.1.4]{Chenavier_2025} for
    counter-examples of the converse implication.
\end{remark}

\subsection{Attractivity of normal forms in metric
spaces}\label{ssct:attractivity-of-nf}

We described in Subsection~\ref{ssct:topological-rewriting} the basic
definitions of topological rewriting theory, including the notion of
\emph{normal forms} of a system. As we will see in
Subsection~\ref{ssct:rew-on-formal-power-series}, our primary example of
a non-discrete topological rewriting system, based on formal power
series, enjoys the property that the distance between any arbirtary starting
formal power series $f$ and any normal form of the system will always be greater
or equal to the distance between any of the successors of $f$ and the considered
normal form. In other words, rewriting cannot take us further away from any
normal forms than we already were. Intuitively, this property seems to describe
a certain kind of attractivity of normal forms; we formalise that concept in
terms of topological rewriting systems for which the underlying space is a
metric space.

\begin{definition}\label{def:attractivity-nf}
    Let $(X, \tau_\dist, \rew)$ be a topological rewriting system where $(X,
    \tau_\dist)$ is the topological space induced by a metric space $(X,
    \dist)$. We say that the system has \emph{globally attractive normal forms}
    if, for any normal form $c$ of the system and any $a, b \in X$ such that $a
    \topRew b$, we have $\dist(b, c) \leq \dist(a, c)$.
\end{definition}

\begin{remark}\label{rmk:charac-attractivity-nf}
    Note how this is equivalent to saying:
    \begin{equation}\label{eq:one-step-attractivity-nf}
        \text{if $a \rew b$, then $\dist(b, c) \leq \dist(a, c)$, for any $a, b
        \in X$ and any normal form $c$}.
    \end{equation} 
    Indeed, the necessity is immediate since $\rew$ is a subrelation of
    $\topRew$. Now, assume (\ref{eq:one-step-attractivity-nf}). Consider $a, b
    \in X$ such that $a \topRew b$ and let $c$ be a normal form of the system.
    First, we show that for any $\varepsilon \in \Rpos$, there exists $d \in X$
    such that $\dist(d, b) \leq \varepsilon$ and $\dist(d, c) \leq d(a, c)$:
    since we have $a \topRew b$ and the open ball $\B_\varepsilon(b)$ centered
    on $b$ of radius $\varepsilon$ is a neighbourhood of $b$, there exists
    $d_\varepsilon \in \B_\varepsilon(b)$ such that $a \transReflRew
    d_\varepsilon$.  Decomposing that transitive reflexive relation into a
    sequence of one-step reduction relations, we obtain, by induction and by
    (\ref{eq:one-step-attractivity-nf}), the following inequality
    $\dist(d_\varepsilon, c) \leq \dist(a, c)$. Hence, for any $\varepsilon \in
    \Rpos$ we have:
    \[
        \dist(b, c) \leq \dist(b, d_\varepsilon) + \dist(d_\varepsilon, c) \leq
        \varepsilon + \dist(a, c).
    \]
    Making $\varepsilon$ tend towards $0$ we obtain the desired result:
    $\dist(b, c) \leq \dist(a, c)$.
\end{remark}

In the next subsection, we will use this newly-defined property to prove that,
under certain assumptions, finitary topological confluence implies the normal
form property with respect to any subrelation of $\topRew$.

\subsection{When finitary topological confluence implies uniqueness of
normal forms}\label{ssct:fin-top-conf-implies-unique-nf}

Throughout this subsection, fix $(X, \tau_\dist, \rew)$ a topological rewriting
system where the underlying topological space $(X, \tau_\dist)$ is induced by a
metric space $(X, \dist)$ and $\squigRew$ a subrelation of $\topRew$. Denote as
usual $\squigTransReflRew$ the transitive reflexive closure of $\squigRew$.

The main theorem of the current subsection is stated as follows:
\begin{theorem}\label{thm:fin-top-conf-implies-unique-nf}
    Assume the following properties:
    \begin{enumerate}
        \item\label{asmpt:transitive} The relation $\topRew$ is transitive.
        \item\label{asmpt:normalising} The system is normalising with respect to
            $\squigRew$.
        \item\label{asmpt:attractivity} The system has globally attractive
            normal forms.
    \end{enumerate}
    Then, if the system is finitary topologically confluent, then it has the
    unique normal form property with respect to $\squigRew$.
\end{theorem}
\begin{proof}
    Since $(X, \tau_\dist)$ is a metrisable space, it is Hausdorff and hence a
    $T_1$-space as well. With that in mind and using
    Assumption~\ref{asmpt:normalising}, by
    Proposition~\ref{prop:equivalence-of-nf-properties}-\ref{item:unique-un},
    it suffices to show that the system has unique normal forms reached by
    $\squigRew$ (Definition~\ref{def:nf-props}-\ref{item:unique-nf}) to prove
    our theorem. Thus, consider $x \in X$ and $\alpha, \alpha'$ be two normal
    forms such that $x \squigTransReflRew \alpha$ and $x \squigTransReflRew
    \alpha'$. Our objective is $\alpha = \alpha'$, let us prove it by
    showing that $\dist(\alpha, \alpha')$ is bounded from above by a quantity
    that converges to $0$.  Since $\squigRew$ is a subrelation of $\topRew$ and
    by Assumption~\ref{asmpt:transitive} that latter relation is transitive, the
    relation $\squigTransReflRew$ is a subrelation of $\topRew$ as well. Hence,
     $x \topRew \alpha$ and $x \topRew \alpha'$. Now, by
    Proposition~\ref{prop:top-rew-charac}, since the open balls
    $\B_\rho(\alpha)$ and $\B_\rho(\alpha')$ of radius $\rho \in \Rpos$ and centered at
    $\alpha$ and $\alpha'$ are neighbourhoods of $\alpha$ and
    $\alpha'$ respectively, there exists $(y, y') \in \B_\rho(\alpha) \times
    \B_\rho(\alpha')$ such that $x \transReflRew y$ and $x \transReflRew y'$. By
    the finitary topological confluence hypothesis, it follows that there exists
    $z \in X$ such that $y \topRew z$ and $y' \topRew z$. Now, by
    Assumption~\ref{asmpt:attractivity}, we have the inequalities $\dist(z,
    \alpha) \leq \dist(y, \alpha)$ and $\dist(z, \alpha') \leq \dist(y',
    \alpha')$. Therefore, we have:
    \[
        \dist(\alpha, \alpha') \leq \dist(\alpha, z) + \dist(z, \alpha') \leq
        \dist(y, \alpha) + \dist(y', \alpha') \leq 2\rho.
    \]
    Hence, making $\rho$ tend towards $0$ we obtain $\alpha = \alpha'$, the
    desired result.
\end{proof}

For our purposes of Subsection~\ref{ssct:cong-ideal}, we will require
the following corollary:
\begin{corollary}\label{cor:fin-top-conf-implies-nf-prop}
    Under the same
    Assumptions~\ref{asmpt:transitive},\ref{asmpt:normalising},\ref{asmpt:attractivity}
    of Theorem~\ref{thm:fin-top-conf-implies-unique-nf}, if the system is
    finitary topologically confluent, then it
    has the normal form property with respect to the relation $\squigRew$.
\end{corollary}
\begin{proof}
    This is a direct consequence of
    Theorem~\ref{thm:fin-top-conf-implies-unique-nf} and
    Statement~\ref{item:un-nf} of
    Proposition~\ref{prop:equivalence-of-nf-properties}.
\end{proof}

\section{Characterisations of standard bases in formal power
series ideals}\label{sct:charac-standard-bases}

In this section, we will first recall in
Subsection~\ref{ssct:rew-on-formal-power-series} how rewriting with respect to
an ideal of formal power series works. This is essentially the same
process as multivariate polynomial reduction but with respect to the opposite
order of an admissible monomial order. We will also prove some
rewriting-theoretic properties of the induced system. Then, in
Subsection~\ref{ssct:standard-representations}, we will introduce the language
which has originally been used to characterise standard
bases~\cite{Becker_1990a}, namely \emph{standard representations} and
\emph{Hironaka remainder}.  Finally, in
Subsection~\ref{ssct:cong-ideal}, we will extend those previous
characterisations by adding some equivalent statements in terms of topological
rewriting theory. In particular, we will recover the result
of~\cite{Chenavier_2025} stating that finitary and infinitary topological
confluences are equivalent for formal power series as well as the
result of~\cite{Chenavier_2020} characterising standard bases with finitary
topological confluence (called $\delta$-confluence in that paper). Moreover, it
will be made clear that the congruence relation modulo the ideal is nothing else
than the equivalence relation generated by the topological rewriting relation
induced by the topological rewriting system from
Subsection~\ref{ssct:rew-on-formal-power-series}.

Throughout
Subsections~\ref{ssct:rew-on-formal-power-series}-\ref{ssct:standard-representations}-\ref{ssct:cong-ideal},
$\K$ is a (commutative) field of arbitrary characteristic, $n$ is a positive
integer,$\DOTS{x_1}{x_n}$ are distinct indeterminates and, as defined in
Subsection~\ref{ssct:formal-power-series}, denote by $[\X]$ (or $[\x]$ for
short) the commutative free monoid generated by $\{\X\}$ and by $\KXX$ (or
$\Kxx$ for short) the $\K$-algebra of formal power series. Consider also a fixed
admissible monomial order $<$ on $[\x]$. Recall from
Subsection~\ref{ssct:formal-power-series}, that $\Kxx$ is a complete local
equicharacteristic Noetherian ring whose $I(\x)$-adic topology is metrisable by
the metric $\delta$ defined in~(\ref{eq:delta}). In particular, the pair $(\Kxx,
\delta)$ is a complete metric space. Let $\tau_\delta$ be the induced topology
by the metric $\delta$. Finally, fix a set $R := \{\DOTS{s_1}{s_r}\} \subseteq
\Kxx \setminus \{0\}$ of exactly $r$ non-zero formal power series, where $r$ is
a non-negative integer and denote by $I(R)$ the ideal in $\Kxx$ generated by $R$.

\subsection{Rewriting on formal power
series}\label{ssct:rew-on-formal-power-series}

Define the following binary relation: for $f, g \in \Kxx$, we write $f \rew g$
if there exist a monomial $M \in \supp{f}$, an index $i \in \{\DOTS{1}{r}\}$ and
$m \in [\x]$ such that $M = m \cdot \LM{\op{<}}{s_i}$ and:
\[
    g = f - \frac{\coeff{f}{M}}{\LC{\op{<}}{s_i}} m s_i.
\]

Notice how that means that $g = (f - \coeff{f}{M}M) + (\text{terms} > M)$. In
other words, $f$ and $g$ have the same coefficients for any monomial that is
strictly smaller than $M$, and the coefficient for $M$ in $g$ is necessarily
zero.

The relation $\rew$ thus depends on $R$ and $<$. We can now consider the
topological rewriting system $\sys := (\Kxx, \tau_\delta, \rew)$ and, therefore,
the topological rewriting relation
(Definition~\ref{def:topological-rewriting-relation}) of that system that we
will denote $\topRew$. It is clear that normal forms
(Definition~\ref{def:normal-form}) of the system $\sys$ are exactly the formal
power series $f$ such that no monomial in $\supp{f}$ is a multiple of a monomial
of the form $\LM{\op{<}}{s_i}$ for any index $i \in \{\DOTS{1}{r}\}$. Denote by
$\NF{\sys}$ the set of normal forms of the system. Let us characterise the
topological rewriting relation in terms of converging sequences.

\begin{lemma}\label{lem:charac-top-rew-fps}
    For any $f, g \in \Kxx$, we have:
    \[
        f \topRew g \qquad\iff\qquad \exists (h_k)_{k \in \N} \in
        \Kxx^{\N}, \begin{cases}
            f \transReflRew h_{k} & \forall k \in \N, \\
            \lim_{k \to \infty} h_k = g.
        \end{cases}
    \]
\end{lemma}
\begin{proof}
    The right-to-left direction is trivial according to
    Proposition~\ref{prop:top-rew-charac}. Now, consider the countable
    fundamental system of neighbourhoods of $g$ defined as
    $\left(\B_{2^{-k}}(g)\right)_{k \in \N}$, where the notation
    $\B_{2^{-k}}(g)$ is the open ball for the metric $\delta$ of radius $2^{-k}$
    centered at $g$. Then, if $f \topRew g$, there exists, for each $k \in \N$,
    at least one $h_k \in \B_{2^{-k}}(g)$ such that $f \transReflRew h_k$. Now,
    by the axiom of choice, it suffices to make arbitrary choices of such $h_k$
    for each $k \in \N$ to construct a sequence that satisfies the right-hand
    side of the equivalence of the lemma.
\end{proof}
\begin{remark}\label{rmk:first-countable}
    Note that this characterisation can be applied \textit{mutatis mutandis} to
    any topological rewriting system whose underlying topological space is
    \emph{first-countable}, that is to say, every element of the space has a
    countable fundamental system of neighbourhoods.
\end{remark}

Here is one important property of the system $\sys$:
\begin{proposition}\label{prop:topRew-transitive}
    If the order $<$ is compatible with the degree, then $\topRew$ is
    transitive.
\end{proposition}
\begin{proof}
    Let $f, g, h \in \Kxx$. We will proceed in three steps.
    \begin{itemize}
        \item  First, we will show that if $f \topRew g \rew h$, then $f \topRew
            h$.
        \item Second, it will follow that if $f \topRew g \transReflRew h$,
            then $f \topRew h$.
        \item Finally, we will use that intermediate result to prove that, if $f
            \topRew g \topRew h$, we indeed have $f \topRew h$, hence $\topRew$
            is transitive.
    \end{itemize}

    Thus, assume firstly that $f \topRew g \rew h$. By definition of $g \rew h$,
    there exist $M \in \supp{g}$, an index $i \in \{\DOTS{1}{r}\}$ and a
    monomial $m \in [\x]$ such that $M = m \cdot \LM{\op{<}}{s_i}$ and:
    \[
        h = g - \frac{\coeff{g}{M}}{\LC{\op{<}}{s_i}} m s_i.
    \]
    From that equality, it follows that $\LM{\op{<}}{g - h} = M$. Let us use
    Proposition~\ref{prop:top-rew-charac} to show that $f \topRew h$. Let
    $U$ be an open neighbourhood of $h$. If $g$ is in $U$, then there exists $h'
    \in U$ such that $f \transReflRew h'$ since $f \topRew g$, Otherwise, assume
    that $g \notin U$. Therefore, there exists $N_U \in \N$ such that
    $\B_{2^{-N_U}}(h)$ is an open ball contained in $U$ and, thus, does not
    contain $g$. It follows from Proposition~\ref{prop:deg-lm-val}, since the
    order is, by hypothesis, compatible with the degree, that
    $\deg\left(\LM{\op{<}}{g - h}\right) = \deg(M) \leq N_U$. By
    Lemma~\ref{lem:charac-top-rew-fps}, there exists a sequence $(h_k)_{k \in
    \N}$ such that $f \transReflRew h_k$ for all $k \in \N$ and $\lim_{k \to
    \infty} h_k = g$. It follows that there exists $K_{N_U} \in \N$ verifying
    $\delta\left(h_{K_{N_U}}, g\right) < 2^{-N_U}$. Since the order is
    compatible with the degree, we obtain by Proposition~\ref{prop:deg-lm-val}
    again:
    \[
        \deg\left(\LM{\op{<}}{g - h_{K_{N_U}}}\right) > N_U \geq \deg(M) =
        \deg\left(\LM{\op{<}}{g - h}\right).
    \]
    By compatibility with the degree, we obtain $\LM{\op{<}}{g -
    h_{K_{N_U}}} > \LM{\op{<}}{g - h}$ and, therefore, $\coeff{g - h_{K_{N_U}}}{M} =
    0$, thus, $\coeff{h_{K_{N_U}}}{M} = \coeff{g}{M} \neq 0$. We can therefore
    define:
    \[
        h' := h_{K_{N_U}} - \frac{\coeff{h_{K_{N_U}}}{M}}{\LC{\op{<}}{s_i}} m
        s_i
    \]
    which by construction verifies $h_{K_{N_U}} \rew h'$. Now we compute:
    \begin{align*}
        h - h'
            &= h - h_{K_{N_U}} + \frac{\coeff{h_{K_{N_U}}}{M}}{\LC{\op{<}}{s_i}} m
            s_i \\
            &= \left(h + \frac{\coeff{g}{M}}{\LC{\op{<}}{s_i}}m s_i\right) -
            h_{K_{N_U}} \\
            &= g - h_{K_{N_U}}.
    \end{align*}

    Thus, trivially $\delta\left(h, h'\right) = \delta\left(g,
    h_{K_{N_U}}\right) < 2^{-N_U}$. Hence, $h' \in \B_{2^{-N_U}}(h) \subseteq
    U$. Moreover, by construction, we have $f \transReflRew h_{K_{N_U}} \rew
    h'$. Therefore, we did exhibit $h' \in U$, for $U$ arbitrary open
    neighbourhood of $h$, such that $f \transReflRew h'$. Thus, $f \topRew h$.

    Secondly, by a simple induction argument, if $f \topRew g \transReflRew h$,
    we have $f \topRew h$ since we decompose $g \transReflRew h$ into a finite
    sequence of one-step reduction relations and thus can apply step-by-step the
    first result of the proof.

    Thirdly, assume $f \topRew g \topRew h$. By
    Lemma~\ref{lem:charac-top-rew-fps}, there exists a sequence $(h_k)_{k \in
    \N}$ such that, for any $k \in \N$, $g \transReflRew h_k$ and
    $\lim_{k \to \infty} h_k = h$. Let us show that $f \topRew h$ by proving
    that, for any $\varepsilon \in \Rpos$, there exists $h' \in
    \B_{\varepsilon}(h)$ such that $f \transReflRew h'$ which enables to
    conclude via Proposition~\ref{prop:top-rew-charac}. Let $\varepsilon \in
    \Rpos$. Then, since $(h_k)_{k \in \N}$ converges to $h$, there exists
    $K_\varepsilon \in \N$ such that $\delta\left(h, h_{K_{\varepsilon}}\right)
    < \frac{\varepsilon}{2}$. Moreover, we have $f \topRew g \transReflRew
    h_{K_{\varepsilon}}$. Hence, by the second result of this proof, we have $f
    \topRew h_{K_{\varepsilon}}$. Using Lemma~\ref{lem:charac-top-rew-fps}
    again, there exists a sequence $(h'_k)_{k \in \N}$ such that, for all $k \in
    \N$, $f \transReflRew h'_k$ and $\lim_{k \to \infty} h'_k =
    h_{K_\varepsilon}$. By that latter convergence property, there exists
    $K'_\varepsilon \in \N$ such that $\delta\left(h_{K_\varepsilon},
    h'_{K'_\varepsilon}\right) < \frac{\varepsilon}{2}$. Hence:
    \[
        \delta\left(h, h'_{K'_\varepsilon}\right) \leq \delta\left(h,
            h_{K_\varepsilon}\right) + \delta\left(h_{K_\varepsilon},
        h'_{K'_\varepsilon}\right) < \frac{\varepsilon}{2} +
        \frac{\varepsilon}{2} = \varepsilon.
    \]
    Finally, by definition of the sequence $(h'_k)_{k \in \N}$, we have $f
    \transReflRew h'_{K'_\varepsilon}$. Which concludes the proof that $f
    \topRew h$.
    
\end{proof}

We will now prove a useful lemma.
\begin{lemma}\label{lem:cauchy-rewriting-seq}
    Let $(f_k)_{k \in \N}$ be a sequence in $\Kxx$. Assume the monomial order $<$ is
    compatible with the degree. If, for all $k \in \N$, we have $f_k
    \rew f_{k+1}$, then the sequence $(f_k)_{k \in \N}$ is convergent in $\Kxx$.
\end{lemma}
\begin{proof}
    For each $k \in \N$, take the monomial $M_k \in \supp{f_k}$, the index
    $i_k \in \{\DOTS{1}{r}\}$ and the monomial $m_k \in [\x]$ such that $M_k =
    m_k \cdot \LM{\op{<}}{s_{i_k}}$ and $f_{k} - f_{k+1} =
    \frac{\coeff{f_k}{M_k}}{\LC{\op{<}}{s_{i_k}}} m_k s_{i_k}$. 

    Firstly, let us show that, for any $k_1 \in \N$, there exists $k' > k_1$
    such that $M_{k_1} < M_k$ for all $k \geq k'$. Let $k_1 \in \N$. Let $M :=
    M_{k_1}$ and define $L_1 := \ens{M_k}{k > k_1, M_k \leq M}$. Since there are
    only finitely many indeterminates and the order is compatible with the
    degree, $L_1$ is finite because $M$ is an upper bound. If $L_1$ is empty,
    then all the subsequent $M_k$ for $k > k_1$ are then bigger than $M$ which
    is the desired result. Assume thus $L_1$ non-empty and denote by $M_1^{\min}
    := \min L_1$ and $k_2 := \min \ens{k \in \N}{k > k_1, M_{k} = M_1^{\min}}$.
    By definition, $M_1^{\min}$ is the monomial in $\supp{f_{k_2}}$ that is
    reduced at step $k_2$. Hence, it follows that $\coeff{f_{k}}{M_1^{\min}} =
    0$ for all $k > k_2$ since $M_1^{\min}$ is the minimum of the monomials
    reduced after step $k_1$. Then, consider the subset $L_2 := \ens{M_{k}}{k >
    k_2, M_k \leq M} \subseteq L_1$ and notice how $M_1^{\min} \notin L_2$ by
    the previous discussion. Repeat this process \textit{ad infinitum} replacing
    the index $j$ by $j+1$ in the definitions of $L_{j+1}$, $M_{j+1}^{\min}$ and
    $k_{j+2}$. Notice that, at each step $j$, we have $L_{j+1} \subsetneq L_j$
    because $M_j^{\min} \notin L_{j+1}$.  But since $L_1$ is finite, there
    necessarily exists $j \geq 1$ such that we have $L_{j+1} := \ens{M_k}{k >
    k_{j+1}, M_k \leq M_{k_1}} = \emptyset$.  Finally, this implies that
    $M_{k_{j+1} + \ell} \notin L_1$, and thus, $M_{k_{j+1} + \ell} > M_{k_1}$,
    for all $\ell \geq 1$.

    Now, let $D \in \N$. Since the order is compatible with the degree and using
    the previous part of this proof, the sequence $(\deg(M_k))_{k \in \N}$ is
    eventually bigger than $D$, \ie, there exists a rank $k_D \in \N$ such that
    $\deg(M_{k}) > D$ for all $k \geq k_D$.  But, a straightforward computation
    shows that, for $k_1, k_2 \in \N$ with $k_D \leq k_1 < k_2$:
    \[
        \deg\left(\LM{\op{<}}{f_{k_1} - f_{k_2}}\right) = \min
        \ens{\deg(M_k)}{k_1 \leq k < k_2} > D.
    \]
    By compatibility with the degree, we use Proposition~\ref{prop:deg-lm-val}
    to show that $\delta(f_{k_1}, f_{k_2}) < 2^{-D}$ which in turn proves that
    $(f_k)_{k \in \N}$ is Cauchy. Now, since $\Kxx$ is complete as metric space,
    it follows that the sequence $(f_k)_{k \in \N}$ does indeed converge to a
    limit in $\Kxx$.
\end{proof}

\begin{remark}\label{rmk:cauchy-rewriting-sequences}
    Denote by $\reflRew$ the reflexive closure of $\rew$, \ie, $f \reflRew g$ if
    and only if $f = g$ or $f \rew g$. As sequences that are eventually
    stationary are Cauchy sequences in a trivial way, any sequence $(f_k)_{k \in
    \N}$ such that $f_k \reflRew f_{k+1}$ (and thus, $f_k \transReflRew
    f_{k+1}$) for all $k \in \N$ is convergent in $\Kxx$, by
    Lemma~\ref{lem:cauchy-rewriting-seq}.
\end{remark}

Now, consider the following relation that we will call the \emph{topological
rewriting with chains relation}:
\[
    f \topChainsRew g \qquad\text{when}\qquad \exists (h_k)_{k \in \N} \in
    \Kxx^{\N}, \begin{cases}
        h_0 = f \\
        h_k \reflRew h_{k+1} & \forall k \in \N, \\
        \lim_{k \to \infty} h_k = g,
    \end{cases}
\]
By Lemma~\ref{lem:charac-top-rew-fps}, it is immediate that $\topChainsRew$ is a
subrelation of $\topRew$, because, if two elements are linked by a finite
sequence of $\reflRew$, then they are related through $\transReflRew$. However,
it is still an open question if $\topChainsRew$ is actually equal to $\topRew$
in the system $\sys$: this is the content of the problem we
call the \emph{chains conjecture}.

\begin{proposition}\label{prop:topChainsTransRew-normalising}
    If the order $<$ is compatible with the degree, then, for all $f \in \Kxx$,
    there exists a normal form $\alpha \in \NF{\sys}$ such that $f \topChainsRew
    \alpha$. In particular, the system $\sys$ is normalising with respect to
    $\topChainsRew$.
\end{proposition}
\begin{proof}
    For any $f \in \Kxx$, denote by $R_f := \ens{m \cdot \LM{\op{<}}{s_i} \in
    \supp{f}}{m \in [\x], i \in \{\DOTS{1}{r}\}}$. This is the set of reducible
    monomials in the support of $f$. Obviously, we have $f \in \NF{\sys}$ if and
    only if $R_f = \emptyset$. Let $f \in \Kxx$. Let us construct a sequence
    $(h_k)_{k \in \N}$ inductively as follows. Start by setting $h_0 := f$. Now,
    let $K \in \N$ and assume we constructed a sequence $h_0 \reflRew h_1
    \reflRew \cdots \reflRew h_K$ such that, for all $k \in \{\DOTS{0}{K-1}\}$,
    if neither $h_k$ nor $h_{k+1}$ are normal forms, \ie{} if we have $R_{h_k}
    \neq \emptyset \neq R_{h_{k+1}}$, then $\min_< R_{h_k} < \min_<
    R_{h_{k+1}}$.  If $R_{h_K} = \emptyset$, define all the subsequent terms of
    the sequence as $h_{K+\ell} := h_K$ for all $\ell \geq 1$. Now, assume that
    $R_{h_K} \neq \emptyset$. Let $M_K := \min_< R_{h_K}$. It exists since the
    order is admissible. Choose $i_K := \min \ens{i \in \{\DOTS{1}{r}\}}{\exists
    m \in [\x], M_K = m \cdot \LM{\op{<}}{s_{i}}}$ and $m_K :=
    \frac{M_K}{\LM{\op{<}}{s_{i_K}}}$. Define:
    \[
        h_{K+1} := h_K - \frac{\coeff{h_K}{M_K}}{\LC{\op{<}}{s_{i_K}}} m_K
        s_{i_K}.
    \]
    Notice how $h_K \rew h_{K+1}$ and, if $R_{h_{K+1}} \neq \emptyset$, then
    $\min_< R_{h_K} < \min_< R_{h_{K+1}}$ because we rewrote the smallest
    reducing monomial in $\supp{h_K}$ and thus, $\coeff{h_{K+1}}{M_K} = 0$ and
    every new potential reducible monomial introduced by the rewriting process
    is necessarily bigger than $M_K$.

    In all cases, repeating this process \textit{ad infinitum} yields a sequence
    $(h_k)_{k \in \N}$ that satisfies the relations $h_0 = f$ and $h_k \reflRew
    h_{k+1}$ for all $k \in \N$. Now, by compatibility with the degree we can
    use Lemma~\ref{lem:cauchy-rewriting-seq} and
    Remark~\ref{rmk:cauchy-rewriting-sequences}. Hence, denote the limit of
    $(h_k)_{k \in \N}$ by $g := \lim_{k \to \infty} h_k$. By definition, we
    have $f \topChainsRew g$. Finally, notice how $g$ is a normal form: indeed,
    if $R_g \neq \emptyset$, say there exists $M := m \cdot \LM{\op{<}}{s_i} \in
    \supp{g}$ for some $m \in [\x]$ and $i \in \{\DOTS{1}{r}\}$ and let $D :=
    \deg(M)$. Since $g$ is the limit of the sequence, there exists $K \in \N$
    such that we have $\delta\left(h_k, g\right) < 2^{-D}$ for all $k \geq K$.
    By compatibility with the degree and Proposition~\ref{prop:deg-lm-val}, we
    obtain, for all $k \geq K$, $\LM{\op{<}}{h_k - g} > M$ and hence
    $\coeff{h_k}{M} = \coeff{g}{M} \neq 0$, so $M \in R_{h_k}$. But, since the
    monomial order $<$ is admissible and compatible with the degree and the
    fact that we have only finitely many indeterminates, for any monomial $m \in
    [\x]$, there are only finitely many monomials that are smaller than $m$ for
    $<$. However, by construction, the sequence $\left(\min_< R_{h_k}\right)_{k
    \in \N}$ is strictly increasing for $<$. Now, since we saw that our
    assumption entails that, for all $k \geq K$, we have $M \in R_{h_k}$ and,
    thus, $\min_< R_{h_k} < M$, this would mean that there are infinitely many
    monomial smaller than $M$ for $<$, because the sequence $\left(\min_>
    R_{h_k}\right)_{k \in \N}$ is strictly increasing but bounded from above by
    $M$. This is a contradiction, hence $g$ is a normal form.
\end{proof}

To conclude this section, we will prove, as we previously said, that the system
$\sys$ has globally attractive normal forms
(Definition~\ref{def:attractivity-nf}) when the order is compatible with the
degree.

\begin{proposition}\label{prop:Kxx-attractive-nf}
    If the order $<$ is compatible with the degree, the system $\sys$ has
    globally attractive normal forms.
\end{proposition}
\begin{proof}
    Let $f, g \in \Kxx$ and $\alpha \in \NF{\sys}$. By
    Remark~\ref{rmk:charac-attractivity-nf}, it suffices to show that, if $f
    \rew g$, then $\delta(g, \alpha) \leq \delta(f, \alpha)$. But, by
    definition, if $f \rew g$, there exists $M \in \supp{f}$, $i \in
    \{\DOTS{1}{r}\}$ and $m \in [\x]$ such that $M = m \cdot \LM{\op{<}}{s_i}$
    and:
    \[
        g = f - \frac{\coeff{f}{M}}{\LC{\op{<}}{s_i}} m s_i.
    \]
    If $g = \alpha$, the result is trivial since $\delta(g, \alpha) = 0$. Assume
    thus $g \neq \alpha$. Then we first want to show that $\LM{\op{<}}{g -
    \alpha} \geq \LM{\op{<}}{f - \alpha}$. Indeed, let $m' \in [\x]$ with $m' <
    \LM{\op{<}}{f - \alpha}$, then, by definition of leading monomial,
    $\coeff{f}{m'} = \coeff{\alpha}{m'}$. Hence, if $m' \in \supp{f}$, it is
    irreducible. By contrapositive, we deduce that $\LM{\op{<}}{f - \alpha} \leq
    M$, since $M$ is reducible and in $\supp{f}$ by definition. But, since $M$
    is the reduced monomial in the rewrite step $f \rew g$, it follows that, for
    all $m' < M$, we have $\coeff{g}{m'} = \coeff{f}{m'}$. If we impose
    furthermore that $m' < \LM{\op{<}}{f - \alpha}$, we get $\coeff{g -
    \alpha}{m'} = \coeff{f - \alpha}{m'} = 0$, which is what we wanted to first
    prove. Now, since the order is assumed to be compatible with the degree, we
    can use Proposition~\ref{prop:deg-lm-val} to conclude that $\delta(g,
    \alpha) \leq \delta(f, \alpha)$.
\end{proof}

\subsection{Standard representations}\label{ssct:standard-representations}

In this subsection, we present the language of standard representations in which
standard bases were originally characterised.

\begin{definition}\label{def:standard-rep}
    Let $f \in \Kxx \setminus \{0\}$. We say that $f$ admits a \emph{standard
    representation with respect to $R$ and $<$} if there exists a $r$-tuple
    $(\DOTS{q_1}{q_r}) \in \Kxx^r$ that satisfies the following two conditions:
    \begin{enumerate}
        \item The $q_i$'s are cofactor coefficients for $f$ with respect to $R$,
            \ie:
            \[
                f = \sum_{i = 1}^{r} q_i s_i.
            \]
        \item There are no cancellations of the least leading terms of the $q_i
            s_i$'s, \ie:
            \[
                \LM{\op{<}}{f} = \min_{<} \ens{\LM{\op{<}}{q_i s_i}}{1 \leq i
                \leq r, q_i \neq 0}.
            \]
    \end{enumerate}
\end{definition}

\begin{remark}\label{rmk:standard-rep-reducible}
    It is evident from the definition that any non-zero formal power series $f$
    that admits a standard representation with respect to $R$ and $<$ is in the
    ideal $I(R)$ and has a leading monomial $\LM{\op{<}}{f}$ which is reducible
    in the system $\sys$. Indeed, choose arbitrarily among the rules in $R$ that
    reduce the leading monomial of $f$ for the opposite order of $<$, \ie{}
    choose any $i_0 \in \ens{i \in \{\DOTS{1}{r}\}}{q_i \neq 0,
    \LM{\op{<}}{q_is_i} = \LM{\op{<}}{f}}$, then:
    \[
        \LM{\op{<}}{f} = \LM{\op{<}}{q_{i_0}s_{i_0}} = \LM{\op{<}}{q_{i_0}} \cdot
        \LM{\op{<}}{s_{i_0}}.
    \]
\end{remark}

It is known from \textsc{Hironaka}'s theorem (see for
instance~\cite[Corollary 2.2]{Becker_1990a}) that $R$ is a standard basis of
$I(R)$ for $<$ if, and only if, every non-zero $f$ in $I(R)$ has a standard
representation with respect to $R$ and $<$. We shall recover that result in
Subsection~\ref{ssct:cong-ideal}, for the case of an order compatible
with the degree, using the following proposition.
\begin{proposition}\label{prop:standard-rep-standard-basis}
    Assume the order $<$ is compatible with the degree. If $f \in \Kxx \setminus
    \{0\}$ is such that $f \topChainsRew 0$, then $f$ admits a standard
    representation with respect to $R$ and $<$.
\end{proposition}
\begin{proof}
    By definition of $f \topChainsRew 0$, there exists a sequence $(h_k)_{k \in
    \N}$ in $\Kxx$ such that $h_0 = f$, for any $k \in \N$, we have $h_k
    \reflRew h_{k+1}$ and finally, $\lim_{k \to \infty} h_k = 0$. Let us
    construct inductively an $r$-tuple of sequences
    $\left(\DOTS{q_1^{(k)}}{q_r^{(k)}}\right)_{k \in \N}$ in $\Kxx$.

    The base step is to set $q_i^{(0)} := 0$ for any $i \in \{\DOTS{1}{r}\}$.
    Now, suppose that we have inductively constructed the sequences up to a rank
    $k \in \N$ such that they satisfy the following condition:
    \[
        f = h_k + \sum_{i = 1}^{r} q_i^{(k)} s_i.
    \]
    If $h_k = h_{k+1}$, set $q_i^{(k+1)} := q_i^{(k)}$ for all $i \in
    \{\DOTS{1}{r}\}$. Otherwise, since $h_k \reflRew h_{k+1}$, we necessarily
    have $h_k \rew h_{k+1}$, which means, by definition, there exist
    $M_k \in \supp{h_k}$, an index $i_k \in \{\DOTS{1}{r}\}$ and a monomial $m_k
    \in [\x]$ such that $M_k = m_k \cdot \LM{\op{<}}{s_{i_k}}$ and:
    \[
        h_{k+1} = h_k - \frac{\coeff{h_k}{M_k}}{\LC{\op{<}}{s_{i_k}}} m_k
        s_{i_k}.
    \]
    We then set $q_i^{(k+1)} := q_i^{(k)}$ for all $i \in \{\DOTS{1}{r}\}
    \setminus \{i_k\}$ and:
    \[
        q_{i_k}^{(k+1)} := q_{i_k}^{(k)} +
        \frac{\coeff{h_k}{M_k}}{\LC{\op{<}}{s_{i_k}}} m_k.
    \]
    Note how we then have:
    \begin{align*}
        h_{k+1} + \sum_{i = 1}^{r} q_i^{(k+1)} s_i
            &= h_{k+1} + \left(q_{i_k}^{(k)} +
            \frac{\coeff{h_k}{M_k}}{\LC{\op{<}}{s_{i_k}}} m_k\right)s_{i_k} +
            \sum_{\substack{i = 1 \\ i \neq i_k}}^{r} q_i^{(k)} s_i \\
            &= \left(h_{k+1} + \frac{\coeff{h_k}{M_k}}{\LC{\op{<}}{s_{i_k}}} m_k
            s_{i_k}\right) + \sum_{i = 1}^{r} q_i^{(k)} s_i \\
            &= h_k + \sum_{i = 1}^{r} q_i^{(k)} s_i \\
            &= f.
    \end{align*}
    Hence the induction hypothesis is verified for $k+1$ and we can continue.
    Repeating this process \textit{ad infinitum} we obtain $r$ sequences
    $\left(q_i^{(k)}\right)_{k \in \N}$ which turn out to be Cauchy. Indeed, we
    notice that, by compatibility with the degree, the sequence $(\deg(m_k))_{k
    \in \N}$ is always eventually bigger than any positive integer because so is
    the sequence $(\deg(M_k))_{k \in \N}$ of reduced monomials in the Cauchy
    rewriting sequence $(h_k)_{k \in \N}$. Since $\Kxx$ is a complete metric
    space, there exists, for each $i \in \{\DOTS{1}{r}\}$, $q_i^{(\infty)} \in
    \Kxx$ such that $\lim_{k \to \infty} q_i^{(k)} = q_i^{(\infty)}$. The
    induction hypothesis allows us to write:
    \[
        f = \lim_{k \to \infty} \left(h_k + \sum_{i = 1}^{r} q_i^{(k)}
        s_i\right) = 0 + \sum_{i = 1}^{r} q_i^{(\infty)} s_i.
    \]
    Because $f$ rewrites into $0$ via chains, there exists a unique $k \in \N$
    such that $M_k = \LM{\op{<}}{f}$. It follows that $q_{i_k}^{(\infty)} \neq
    0$ and $\LM{\op{<}}{q_{i_k}^{(\infty)} s_{i_k}} = M_k = \LM{\op{<}}{f}$.
    Hence, the $q_i^{(\infty)}$'s do exhibit a standard representation for $f$
    with respect to $R$ and $<$.
\end{proof}

\subsection{The congruence relation modulo the ideal and the characterisations
of standard bases}\label{ssct:cong-ideal}

In this subsection, we will show that the congruence relation modulo the ideal
$I(R)$ is exactly the equivalence relation generated by $\topChainsRew$, when
the order $<$ is assumed to be compatible with the degree. Under this
assumption, we will also prove that this latter equivalence relation is the same
as the equivalence relation generated by $\topRew$.

Let $\equiv_{I(R)}$ be the congruence relation modulo the ideal $I(R)$,
that is to say, we have $f \equiv_{I(R)} g$ if and only if $f - g \in I(R)$.
Denote by $\topChainsEquivRew$ (resp.\ $\topEquivRew$) the equivalence relation
generated by $\topChainsRew$ (resp.\ $\topRew$): it is the transitive closure of
the symmetric closure, denoted $\topChainsSymRew$ (resp.\ $\topSymRew$), of the
relation $\topChainsRew$ (resp.\ $\topRew$).

We are first going to prove a few lemmas.

\begin{lemma}[{\cite[Proposition 4.2.2]{Chenavier_2025}}]\label{lem:f-rew-g-diff-in-ideal}
    For all $f, g \in \Kxx$, if $f \topRew g$, then $f \equiv_{I(R)} g$.
\end{lemma}

This lemma is true only because, as we mentioned before, ideals of (commutative)
formal power series are topologically closed for the topology $\tau_\delta$.

\begin{lemma}\label{lem:rewrite-rule-to-zero}
    Assume the order $<$ is compatible with the degree. Then, for all $q \in
    \Kxx$ and all $i \in \{\DOTS{1}{r}\}$, we have $qs_i \topChainsRew 0$.
\end{lemma}
\begin{proof}
    Because the order is compatible with the degree and thus of order type
    $\omega$, we can write the support of $q$ as a (potentially infinite)
    sequence $(m_k)_{k \in \N}$ of monomials that is stricly increasing for $<$.
    If that sequence is finite, then it is clear that, after rewriting each
    monomial of the form $M_k := m_k \cdot \LM{\op{<}}{s_i} \in \supp{qs_i}$
    using the rule of index $i$, everything will cancel out and we would have $q
    s_i \transReflRew 0$. It then suffices to decompose that relation in
    one-step reduction relations and complete to infinity with zeros to end up
    with a sequence which satisfies $qs_i \topChainsRew 0$. Otherwise, if the
    support is infinite, we define $h_0 = qs_i$ and, for each $k \in \N$:
    \[
        h_{k+1} := h_k - \frac{\LC{\op{<}}{h_k}}{\LC{\op{<}}{s_i}} m_k s_i.
    \]
    It is straightforward to see that, for every $k \in \N$, we have $h_k \rew
    h_{k+1}$, because we rewrite the leading monomial of $h_k$, which is of the
    form $M_k := m_k \cdot \LM{\op{<}}{s_i}$, with rule of index $i$. Hence, by
    Lemma~\ref{lem:cauchy-rewriting-seq}, the sequence $(h_k)_{k \in \N}$
    converges. Moreover, it is clear that its limit is $0$ because, the sequence
    $(m_k)_{k \in \N}$ being stricly increasing for $<$ and the order $<$ being
    assumed to be compatible with the degree, it follows that the sequence
    $(\deg(M_k))_{k \in \N}$ is eventually bigger than any positive integer and
    the sequence $(M_k)_{k \in \N}$ is exactly the sequence of leading monomials
    of the sequence $(h_k)_{k \in \N}$.
\end{proof}

\begin{lemma}[Translation lemma]\label{lem:translation-lemma}
    Assume the order $<$ is compatible with the degree. For all formal power
    series $f, g, h \in \Kxx$, if $f - g \topChainsRew h$, then there exist
    $f', g' \in \Kxx$ such that $h = f' - g'$, $f \topChainsRew f'$ and $g
    \topChainsRew g'$.
\end{lemma}
\begin{proof}
    By definition, we have a sequence $(h_k)_{k \in \N}$ such that $h_0 = f -
    g$, $\lim_{k \to \infty} h_k = h$ and, for all $k \in \N$, $h_k \reflRew
    h_{k+1}$. Construct inductively two sequences $(f_k)_{k \in \N}$ and
    $(g_k)_{k \in \N}$. First, set $f_0 := f$ and $g_0 := g$. Assume the
    sequences are constructed up to a rank $k \in \N$ such that $h_k = f_k -
    g_k$. If we have $h_k = h_{k+1}$, we set $f_{k+1} := f_k$ and $g_{k+1} :=
    g_k$. Otherwise, if $h_k \neq h_{k+1}$, we necessarily have $h_k \rew
    h_{k+1}$, which means by definition, that there exist $M_k \in \supp{h_k}$,
    an index $i_k \in \{\DOTS{1}{r}\}$ and a monomial $m_k \in [\x]$ such that
    $M_k = m_k \cdot \LM{\op{<}}{s_{i_k}}$ and:
    \[
        h_{k+1} = h_k - \frac{\coeff{h_k}{M_k}}{\LC{\op{<}}{s_{i_k}}}m_ks_{i_k}.
    \]
    We then set:
    \[
        f_{k+1} := f_k - \frac{\coeff{f_k}{M_k}}{\LC{\op{<}}{s_{i_k}}}m_ks_{i_k}
        \qquad\text{and}\qquad
        g_{k+1}:= g_k - \frac{\coeff{g_k}{M_k}}{\LC{\op{<}}{s_{i_k}}}m_ks_{i_k}.
    \]
    Note how, then, by induction, we have:
    \begin{align*}
        h_{k+1}
            &= h_k - \frac{\coeff{h_k}{M_k}}{\LC{\op{<}}{s_{i_k}}}m_ks_{i_k} \\
            &= (f_k - g_k) - \frac{\coeff{f_k -
            g_k}{M_k}}{\LC{\op{<}}{s_{i_k}}}m_ks_{i_k} \\
            &= \left(f_k -
                \frac{\coeff{f_k}{M_k}}{\LC{\op{<}}{s_{i_k}}}m_ks_{i_k}\right)
            - \left(g_k -
                \frac{\coeff{g_k}{M_k}}{\LC{\op{<}}{s_{i_k}}}m_ks_{i_k}\right) \\
            &= f_{k+1} - g_{k+1}.
    \end{align*}
    If $M_k \notin \supp{f_k}$, then $f_k = f_{k+1}$; otherwise, $f_k \rew
    f_{k+1}$ by definition. Hence, $f_k \reflRew f_{k+1}$, for any $k \in \N$.
    Same reasoning applies to the sequence $(g_k)_{k \in \N}$. Since the order
    is compatible with the degree, we can use
    Lemma~\ref{lem:cauchy-rewriting-seq} to conclude that the sequences
    $(f_k)_{k \in \N}$ and $(g_k)_{k \in \N}$ have limits; denote them by $f'$ and
    $g'$ respectively. It is then clear that $f \topChainsRew f'$ and $g
    \topChainsRew g'$. Finally, we conclude with:
    \[
        h = \lim_{k \to \infty} h_k = \lim_{k \to \infty} (f_k - g_k) = \lim_{k
        \to \infty} f_k - \lim_{k \to \infty} g_k = f' - g'.
    \]
\end{proof}

We will actually require only a particular case of that lemma that we state
thereafter without the straightforward proof.
\begin{corollary}\label{cor:translation-lemma}
    Assume the order $<$ is compatible with the degree. Let $f, g \in \Kxx$. If
    we have $f - g \topChainsRew 0$, then there exists $h \in \Kxx$ such that $f
    \topChainsRew h \topChainsInvRew g$.
\end{corollary}

\begin{lemma}\label{lem:cong-subrel-equivrew}
    Assume the order $<$ is compatible with the degree. Then $\equiv_{I(R)}$ is
    a subrelation of $\topChainsEquivRew$.
\end{lemma}
\begin{proof}
    First notice how for any subset $R'$ such that $R \subseteq R' \subseteq \Kxx \setminus
    \{0\}$, the relation $\topChainsRew$ of the system $\sys$ is a subrelation
    of the topological rewriting with chains relation with respect to $R'$ and
    $<$. Now, to get the proposition, it suffices to prove that for all $r \in
    \N$, if there exist $(\DOTS{q_1}{q_r}) \in \Kxx^r$ such that $f - g =
    \sum_{i = 1}^{r} q_i s_i$, then we have $f \topChainsEquivRew g$. Let us now
    prove this by induction on the number $r$ of rules in $R$. The base step is
    when $r = 0$ which means that $f = g$, and since $\topChainsEquivRew$ is an
    equivalence relation it is in particular reflexive; the result follows. By
    induction, let $r \in \N$ satisfy the induction hypothesis and
    consider $s \in \Kxx \setminus (R \cup \{0\})$, $f, g \in \Kxx$ and
    $(\DOTS{q, q_1}{q_{r}}) \in \Kxx^{r+1}$ such that:
    \[
        g - f = q s + \sum_{i = 1}^{r} q_i s_i.
    \]
    Let $\topChainsRew[R_s]$ be the topological rewriting with chains relation
    with respect to $R \cup \{s\}$ and $<$ and denote by
    $\topChainsEquivRew[R_s]$ the equivalence relation generated by it. By
    Lemma~\ref{lem:rewrite-rule-to-zero}, we have $qs \topChainsRew[R_s] 0$.
    Applying Corollary~\ref{cor:translation-lemma} to $f := f$ and $g := f +
    qs$, we obtain a formal power series $h \in \Kxx$ such that
    $f \topChainsRew[R_s] h$ and $f + qs \topChainsRew[R_s] h$. It follows that
    $f \topChainsEquivRew[R_s] f + qs$. Now, notice that, by definition, we have
    $g - (f + qs) = \sum_{i = 1}^{r} q_is_i$. Apply the induction hypothesis to
    obtain $g \topChainsEquivRew f + qs$. But, $\topChainsRew$ is a subrelation
    of $\topChainsRew[R_s]$ (and hence, $\topChainsEquivRew$ is also a
    subrelation of $\topChainsEquivRew[R_s]$), so $g \topChainsEquivRew[R_s] f +
    qs \topChainsEquivRew[R_s] f$, therefore, $f \topChainsEquivRew[R_s] g$,
    which is the desired result.
\end{proof}

We can finally prove that the congruence relation is characterised by the
equivalence relations generated by the topological rewriting relations.
\begin{theorem}\label{thm:decide-ideal}
    Assume the order $<$ is compatible with the degree. Then $\equiv_{I(R)}$,
    $\topEquivRew$ and $\topChainsEquivRew$ are equal.
\end{theorem}
\begin{proof}
    By Lemma~\ref{lem:cong-subrel-equivrew}, $\equiv_{I(R)}$ is a
    subrelation of $\topChainsEquivRew$. Now, since $\topChainsRew$ is a
    subrelation of $\topRew$, it follows that $\topChainsEquivRew$ is also a
    subrelation of $\topEquivRew$. Finally, let us show that $\topEquivRew$ is a
    subrelation of $\equiv_{I(R)}$, which will prove the theorem. Let $f, g \in
    \Kxx$ such that $f \topEquivRew g$. Decompose this into a finite sequence:
    \[
        f \topSymRew h_1 \topSymRew h_2 \topSymRew \cdots \topSymRew h_\ell
        \topSymRew g.
    \]
    Using Lemma~\ref{lem:f-rew-g-diff-in-ideal}, it follows that, for any
    $h, h' \in \Kxx$,  $h \topSymRew h'$ implies $h - h' \in I(R)$. Thus, by
    induction on the finite sequence above, we have:
    \[
        f - g = (f - h_1) + (h_1 - h_2) + \cdots + (h_\ell - g) \in I(R).
    \]
    From which we deduce $f \equiv_{I(R)} g$.
\end{proof}

The following theorem is where the main contributions of this paper can be
found.

\begin{theorem}\label{thm:charac-standard-bases}
    Assume the order $<$ is compatible with the degree. Then the following
    statements are equivalent:
    \begin{enumerate}
        \item\label{sta:fin-topo-conf}
            The system $\sys$ is finitary topologically confluent.
        \item\label{sta:topRew-zero}
            For all $f \in I(R)$, we have $f \topRew 0$.
        \item\label{sta:reducible}
            For all $f \in I(R) \setminus \{0\}$, $f$ is reducible (\ie{} not a
            normal form).
        \item\label{sta:standard-basis}
            The set $R$ is a standard basis of $I(R)$ for $<$.
        \item\label{sta:topChainsRew-zero}
            For all $f \in I(R)$, we have $f \topChainsRew 0$.
        \item\label{sta:standard-rep}
            For all $f \in I(R) \setminus \{0\}$, $f$ admits a standard
            representation with respect to $R$ and $<$.
        \item\label{sta:remainder}
            For all $f \in \Kxx$, there is a unique $\alpha \in \NF{\sys}$ such
            that $f = \alpha + \sum_{i = 1}^{r} q_i s_i$ for some
            $(\DOTS{q_1}{q_r}) \in \Kxx^r$.
        \item\label{sta:normal-forms}
            The map from $\NF{\sys}$ to the quotient $\Kxx/I(R)$ which
            associates to any normal form $\alpha \in \NF{\sys}$ the equivalence class
            $\alpha + I(R)$ is bijective.
        \item\label{sta:infty-topo-conf}
            The system $\sys$ is infinitary topologically confluent.
    \end{enumerate}
\end{theorem}
\begin{proof}
    Most results use the compatibility with the degree assumption.

    \ref{sta:fin-topo-conf}. $\Rightarrow$ \ref{sta:topRew-zero}.: Assuming
    finitary topological confluence of the system and thanks to
    Propositions~\ref{prop:topRew-transitive},~\ref{prop:topChainsTransRew-normalising},
    and~\ref{prop:Kxx-attractive-nf}, we can use
    Corollary~\ref{cor:fin-top-conf-implies-nf-prop} of
    Theorem~\ref{thm:fin-top-conf-implies-unique-nf} to deduce that the normal
    form property with respect to $\topRew$ is verified. Now, let $f \in I(R)$.
    We thus have $f \equiv_{I(R)} 0$. By Theorem~\ref{thm:decide-ideal}, we have
    $f \topEquivRew 0$. By the normal form property, since $0$ is obviously a
    normal form of the system, we have $f \topRew 0$ since $\topRew$ is
    transitive.

    \ref{sta:topRew-zero}. $\Rightarrow$ \ref{sta:reducible}.: Let $f \in I(R)
    \setminus \{0\}$. We have by assumption $f \topRew 0$. By contradiction,
    assume $f$ is a normal form. Then, by Lemma~\ref{lem:charac-normal-forms},
    $f = 0$ which contradicts the definition of $f$.

    \ref{sta:reducible}. $\Rightarrow$ \ref{sta:standard-basis}.: Let $f \in
    I(R) \setminus \{0\}$. Let us show that $\LM{\op{<}}{f}$ is reducible. By
    contradiction, suppose that there is no $i \in \{\DOTS{1}{r}\}$ such that
    $\LM{\op{<}}{s_i}$ divides $\LM{\op{<}}{f}$. By
    Proposition~\ref{prop:topChainsTransRew-normalising}, there exists $\alpha
    \in \NF{\sys}$ such that $f \topChainsRew \alpha$. Since $\topChainsRew$ is
    a subrelation of $\topRew$, we get $f \topRew \alpha$. From
    Lemma~\ref{lem:f-rew-g-diff-in-ideal}, we get $f - \alpha \in I(R)$. But,
    because $f \in I(R)$, we have $\alpha \in I(R)$. Now, note how if $\alpha
    = 0$, then $f \topChainsRew 0$, and therefore $\LM{\op{<}}{f}$ is
    necessarily reducible. Otherwise, we would have $\alpha \in I(R) \setminus
    \{0\}$ and, by assumption, we would deduce that $\alpha$ is not a normal
    form; a contradiction.

    \ref{sta:standard-basis}. $\Rightarrow$ \ref{sta:topChainsRew-zero}.: Let $f
    \in I(R)$. Set $f_0 := f$. If $f_0 = 0$, it is trivial to see that $f
    \topChainsRew 0$. Assume thus $f_0 \neq 0$. Then $\LM{\op{<}}{f_0}$ is
    reducible. Rewrite it to obtain $f_1$. We thus have $f_0 \rew f_1$. If
    $f_1 = 0$, the result is clear. Otherwise, $f_1 \in I(R) \setminus \{0\}$.
    Therefore, its leading monomial is once again reducible. Rewrite it to
    obtain $f_2$, and go on like this forever or until we reach,
    for some $k \in \N$, the identity $f_k = 0$. If that is the case, then it is
    clear that $f \topChainsRew 0$. Otherwise, if there is never any $k \in \N$
    such that $f_k = 0$ then we obtain an infinite sequence $(f_k)_{k \in \N}$
    that satisfies the relations $f_0 = f$ and $f_k \rew f_{k+1}$ for all $k \in
    \N$. By Lemma~\ref{lem:cauchy-rewriting-seq}, it follows that the sequence
    converges. Moreover, since we always rewrite the leading monomial and since
    the order is of type $\omega$, being compatible with the degree, we conclude
    that the limit is $0$ which exactly means that $f \topChainsRew 0$.

    \ref{sta:topChainsRew-zero}. $\Rightarrow$ \ref{sta:standard-rep}.: This is
    exactly the content of Lemma~\ref{prop:standard-rep-standard-basis}.

    \ref{sta:standard-rep}. $\Rightarrow$ \ref{sta:remainder}.: Let $f \in
    \Kxx$. If $f$ is a normal form, the existence part of~\ref{sta:remainder}. is
    guaranteed. Otherwise, by
    Proposition~\ref{prop:topChainsTransRew-normalising}, there exists $\alpha
    \in \NF{\sys}$ different from $f$ such that $f \topChainsRew \alpha$. Now,
    since $\topChainsRew$ is a subrelation of $\topRew$, we obtain by
    Lemma~\ref{lem:f-rew-g-diff-in-ideal} that $f - \alpha \in I(R) \setminus
    \{0\}$. Hence, by assumption,  $f - \alpha$ admits a standard representation
    with respect to $R$ and $<$, which yields the existence of the cofactors
    $(\DOTS{q_1}{q_r}) \in \Kxx^r$ such that $f - \alpha = \sum_{i = 1}^{r} q_i
    s_i$. Let us now show the uniqueness of $\alpha \in \NF{\sys}$ satisfying
    that property. Suppose we have $\beta \in \NF{\sys}$ different from $\alpha$
    and $(\DOTS{q'_1}{q'_r}) \in \Kxx^r$ such that $f = \sum_{i = 1}^{r} q'_i
    s_i + \beta$. Then $\alpha - \beta = \sum_{i = 1}^{r} (q_i - q'_i) s_i \in
    I(R) \setminus \{0\}$. By assumption, $\alpha - \beta$ admits a standard
    representation with respect to $R$ and $<$. Then, according to
    Remark~\ref{rmk:standard-rep-reducible}, it follows that the leading
    monomial of $\alpha - \beta$ is reducible. However, this would mean that,
    among $\alpha$ and $\beta$, at least one of them contains a reducible
    monomial in its support, a contradiction. Hence, $\alpha = \beta$.

    \ref{sta:remainder}. $\Rightarrow$ \ref{sta:normal-forms}.: It is a rather
    straightforward reformulation.

    \ref{sta:normal-forms}. $\Rightarrow$ \ref{sta:infty-topo-conf}.: Let $f, g,
    h \in \Kxx$ such that $f \topRew g$ and $f \topRew h$. By
    Proposition~\ref{prop:topChainsTransRew-normalising}, there exist normal
    forms $\alpha,
    \beta \in \NF{\sys}$ such that $g \topChainsRew \alpha$ and $h \topChainsRew
    \beta$. As $\topChainsRew$ is a subrelation of $\topRew$, we obtain thus that:
    \[
        \alpha \topInvRew g \topInvRew f \topRew h \topRew \beta.
    \]
    It is clear from Lemma~\ref{lem:f-rew-g-diff-in-ideal} that $\alpha - \beta
    = (\alpha - g) + (g - f) + (f - h) + (h - \beta) \in I(R)$. Thus, by
    assumption, $\alpha = \beta$. Hence, the system is infinitary topologically
    confluent, since we have:
    
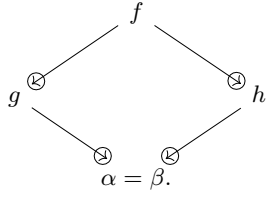
\[
        \begin{tikzcd}
            & \ar[circle]{dl} f \ar[circle]{dr} & \\
            g \ar[circle]{dr} & & \ar[circle]{dl} h \\
                              & \alpha = \beta. &
        \end{tikzcd}    
    \]

    \ref{sta:infty-topo-conf}. $\Rightarrow$ \ref{sta:fin-topo-conf}.: This is
    the content of Remark~\ref{rmk:infty-implies-finitary}.
\end{proof}

\section*{Conclusion}
\addcontentsline{toc}{section}{Conclusion}

In this paper, we demonstrated why \Grobner{} bases fail to work for complete
local equicharacteristic Noetherian rings and we gave convenient
rewriting-theoretic characterisations of standard bases for ideals of
formal power series, which enables to compute in those rings. Further work in
that field could focus on the establishment of a ``topologised'' Newman's lemma
from the classical case. It would then in particular allow us to recover the
analogous statement of the Buchberger criterion for standard bases. Note however
how the desired end result about that S-series criterion has already been proven
in~\cite{Becker_1990}, so it would once again be about having purely
rewriting-theoretic proofs of the already known results. On other perspectives,
there is still the open problem of the \emph{chains conjecture}: given an
admissible monomial order $<$ compatible with the degree and $R$ a finite set
of rewrite rules for formal power series, are the relations $\topRew$ and
$\topChainsRew$ from the topological rewriting system $\sys$ necessarily equal?
This is in general not true for $R$ infinite. If the conjecture turns out to be
true, then the topological rewriting relation we work with in this paper would
be the analogous as the ``strongly convergent transfinite reduction relation''
of infinitary $\Sigma/\lambda$-term rewriting. This would mean that abstract
topological rewriting theory is indeed a theory encompassing both rewriting on
formal power series and infinitary term rewriting in computer science.

\backmatter%

\bmhead{Acknowledgements}

The author would like to thank Cyrille \textsc{Chenavier} and Thomas
\textsc{Cluzeau} for giving helpful feedback and comments both on the research
results presented here and for the writing process of this paper, as well as
Yann \textsc{Denichou} for providing helpful assistance in \LaTeX{} to design
the arrow $\topRew$.


\bibliography{sn-bibliography}

\end{document}